\documentclass[aap]{imsart}

\RequirePackage{amsthm,amsmath,amsfonts,amssymb,mathtools}
\RequirePackage{comment}
\RequirePackage[numbers]{natbib}
\RequirePackage[colorlinks,citecolor=blue,urlcolor=blue]{hyperref}
\RequirePackage{graphicx}

\startlocaldefs

\theoremstyle{plain}\newtheorem{lemma}{\textbf{Lemma}}\newtheorem{theorem}{\textbf{Theorem}}
\theoremstyle{remark}
\newtheorem{definition}{\textbf{Definition}}

\theoremstyle{definition}

\theoremstyle{definition}\newtheorem{remark}{\textbf{Remark}}

\makeatletter
\renewenvironment{proof}[1][\proofname] {
	\par\pushQED{\qed}\normalfont
	\topsep6\p@\@plus6\p@\relax
	\trivlist\item[\hskip\labelsep\bfseries#1\@addpunct{:}]
 	\ignorespaces
} {
	\popQED\endtrivlist\@endpefalse
}
\makeatother

\newtheorem{conjecture}{Conjecture}[section]

\DeclareMathOperator*{\argmin}{arg\,min}

\newcommand{\dkl}{\mathrm{D_{KL} }}

\endlocaldefs

\begin{document}

\begin{frontmatter}
\title{The TAP free energy for high-dimensional linear regression}
\runtitle{TAP free energy for linear regression}

\begin{aug}
\author[A]{\fnms{Jiaze} \snm{Qiu}\ead[label=e1]{jiazeqiu@g.harvard.edu}},
\author[B]{\fnms{Subhabrata} \snm{Sen}\ead[label=e2]{subhabratasen@fas.harvard.edu}}
\address[A]{Department of Statistics, Harvard University, \printead{e1}}

\address[B]{Department of Statistics, Harvard University, \printead{e2}}
\end{aug}

\begin{abstract}
We derive a variational representation for the log-normalizing constant of the posterior distribution in Bayesian linear regression with a uniform spherical prior and an i.i.d. Gaussian design. We work under the ``proportional" asymptotic regime, where the number of observations and the number of features grow at a proportional rate. This rigorously establishes the Thouless-Anderson-Palmer (TAP) approximation arising from spin glass theory, and proves a conjecture of \cite{krzakala2014variational} in the special case of the spherical prior. 
\end{abstract}



\end{frontmatter}


\section{Introduction}
The analysis of high-dimensional probability distributions is a central challenge in modern Statistics and Machine Learning. This is particularly true in the context of Bayesian Statistics, where scientists carry out inference based on the posterior distribution. In modern applications, the posterior distribution is typically high-dimensional, and analytically intractable. Variational Inference (VI) has emerged as an attractive option to approximate these intractable distributions, facilitating fast, parallel computations in state-of-the-art applications \cite{wainwright2008graphical,blei2017variational}. In this approach, the distribution of interest is approximated (in KL divergence) by distributions from a pre-specified, more tractable collection. The simplest version of VI is the Naive Mean-field approximation (NMF), where the distribution of interest is approximated by a product distribution. This strategy has roots in classical approximations for spins-systems in statistical physics \cite{mezard2009information}.

Despite the rapidly growing popularity of variational approximations in Statistics and Machine Learning, the corresponding theoretical guarantees for these approximations are quite limited. In this paper, we study Bayesian linear regression, a standard workhorse of modern Statistics, via the lens of (advanced) Mean-field approximations. Given data $\{(y_i,x_i): 1\leq i \leq n\}$, $y_i \in \mathbb{R}$, $x_i \in \mathbb{R}^p$, the scientist posits a linear model $y_i = x_i^{\top}\beta + \varepsilon_i$, where $\beta \in \mathbb{R}^p$ and $\varepsilon_i \sim N(0, \Delta)$ are iid. Given a prior distribution $\pi_0$, setting $y=(y_1,\cdots,y_n) \in \mathbb{R}^n$ and $X^{\top}=[x_1, \cdots, x_n] \in \mathbb{R}^{p \times n}$,  the scientist constructs the posterior 
\begin{equation}
    \frac{\mathrm{d} \mathbb{P} }{\mathrm{d} \pi_0 } (\beta | y, X) =  \frac{1}{\mathcal{Z}_p} e^{-\frac{1}{2\Delta} \lVert y - X\beta\rVert^2}  \text{,} \label{eq:posterior}
\end{equation}
where $\mathcal{Z}_p$ is the normalizing constant of the posterior distribution. Borrowing terminology from statistical physics, we will refer to $\mathcal{Z}_p$ as the partition function of the model. 

Using the classical Gibbs variational principle \cite{wainwright2008graphical}, one has 
\begin{align}
    \log \mathcal{Z}_p = \sup_{Q \ll \pi_0}\Big( \mathbb{E}_{Q}\Big[-\frac{1}{2\Delta} \|y-X\beta\|_2^2 \Big] - \dkl(Q\| \pi_0) \Big), \nonumber  
\end{align}
with equality if $Q$ is the posterior distribution. Restricting the supremum above to product measures, one obtains the NMF approximation to the log-partition function. Note that the NMF approximation is always a valid lower bound:  
\begin{align}
    \log \mathcal{Z}_p \geq \sup_{Q= \prod_i Q_i} \Big( \mathbb{E}_{Q}\Big[-\frac{1}{2\Delta} \|y-X\beta\|_2^2 \Big] - \dkl(Q\| \pi_0) \Big). \nonumber
\end{align}
We say that the NMF approximation is \emph{correct to leading order} if 
\begin{align}
    \log \mathcal{Z}_p - \sup_{Q= \prod_i Q_i} \Big( \mathbb{E}_{Q}\Big[-\frac{1}{2\Delta} \|y-X\beta\|_2^2 \Big] - \dkl(Q\| \pi_0) \Big) = o(p). \label{eq:nmf_approx} 
\end{align}
In recent joint work with Sumit Mukherjee \cite{mukherjee2021variational}, the second author leveraged recent advances in the theory of non-linear large deviations to derive sufficient conditions on $X$ for the accuracy of NMF approximation when $\pi_0= \prod^{\otimes p} \pi$ is a product distribution. In this setting, it is easy to see that the supremum in \eqref{eq:nmf_approx} is obtained by $Q^*=\prod_i Q_i^*$, with 
$\frac{\mathrm{d}Q_i^*}{\mathrm{d}\pi}(\beta_i) \propto \exp\Big(- \frac{(\beta_i - R_i)^2}{2\Sigma_i^2}\Big)$.
The approximation \eqref{eq:nmf_approx} can thus be equivalently parametrized in terms of $a_i = \mathbb{E}_{Q_i^*}[\beta_i]$ and $c_i = \mathrm{Var}_{Q_i^*}[\beta_i]$, so that 
\begin{align}
    \log \mathcal{Z}_p = \sup_{a, c} \Big( - \frac{1}{2\Delta} \| y - X a\|_2^2 - \dkl (Q^*\| \pi_0)  \Big) + o(p) \label{eq:nmf_prod}  
\end{align}
if the NMF approximation is accurate to leading order.

In the special case where $x_i$'s are iid standard gaussian, \cite{mukherjee2021variational} establishes the accuracy of the NMF approximation as long as $p=o(n)$. This result is expected to be tight---using insights from spin glasses, physicists predict that Bayes linear regression is \emph{not} approximately NMF if the number of samples $n$ and the number of features $p$ grow proportionally \cite{krzakala2014variational}. In fact, using the Thouless-Anderson-Palmer (TAP) formulation, one obtains the following prediction for the log-partition function. 

\begin{conjecture}[\cite{krzakala2014variational}]
\label{conj:TAP_reg} 
Assume $\pi_0 = \prod^{\otimes p} \pi$, where $\pi$ is a probability distribution on $[-1,1]$. Then 
\begin{align}
\log \mathcal{Z}_p = \max_{\|a \|_{\infty} \leq 1} \Big[ -\frac{ \lVert  y - X a \rVert^2}{2\Delta} -\frac{n}{2} \log \Big(1 + \frac{\bar{c}}{\Delta } \Big) + \sum_{i=1}^{p} \log \mathbb{E}_{\pi}\Big[ \exp( - \frac{(\beta_i - R_i)^2}{2 \Sigma_*^2} )\Big] \Big] + o(p),  \label{eq:tapreg_conjecture}
\end{align}
where the parameters $R_i$, $c_i$ and $\Sigma_*$ are specified as follows: define 
$(\mathrm{d} Q_i / \mathrm{d}\pi)(\beta) \propto \exp(- (\beta- R_i)^2 / (2 \Sigma_*^2) )$. $R_i$ is chosen such that $\mathbb{E}_{Q_i}[\beta] = a_i$. Set $c_i = \mathrm{Var}_{Q_i}(\beta)$. Finally, the parameter $\Sigma_*$ depends on $\Delta$ and $\alpha$, and is specified implicitly as the smallest solution to a fixed point equation. 
\end{conjecture}
The interval $[-1,1]$ is arbitrary, the conjecture generalizes directly to any prior $\pi$ supported on a bounded set.  
The main difference between \eqref{eq:nmf_prod} and \eqref{eq:tapreg_conjecture} is the second term in the RHS of \eqref{eq:tapreg_conjecture}. This term is referred to as the \emph{Onsager Correction} term in statistical physics. The representation \eqref{eq:tapreg_conjecture} is crucial for a number of reasons: 
\begin{itemize}
    \item[(i)] It suggests that the widely popular NMF approximation is actually incorrect in certain regimes, and thus practitioners should exercise caution before using the NMF  approximation. 
    \item[(ii)] In settings where the TAP formalism is conjectured to be correct, it is further conjectured that the variational problem \eqref{eq:tapreg_conjecture} has a unique global optimizer. Further, this optimizer is expected to be an asymptotically optimal estimator (in a Bayesian sense).     
    \item[(iii)] This suggests the following approach---optimize the RHS of \eqref{eq:tapreg_conjecture} directly using an out of the box algorithm such as gradient descent, and then use the resulting estimator for inference. 
\end{itemize} 
Exploring the validity of the TAP framework in this problem thus has vital practical implications. From a theoretical perspective, the posterior distribution \eqref{eq:posterior} can be viewed as a \emph{planted} spin glass model. Due to special symmetries present in these models (e.g. the Nishimori identities), posterior distributions naturally provide examples of \emph{replica symmetric} models in spin glass theory. They provide ideal test beds to study the broader applicability of ideas developed to study mean-field spin glass models such as the Sherrington-Kirkpatrick (SK) and mixed p-spin models. Techniques developed originally in the context of mean-field spin glasses have already been utilized to characterize the limiting mutual information between the data and the underlying signal in the Bayesian linear regression model \cite{barbier2020mutual}. In this paper, we adopt a complementary approach, and initiate a study of Bayes linear regression model through the lens of the TAP approximation.


\subsection{Main Result}
Our main result in this paper is the analogue of Conjecture \ref{conj:TAP_reg} under a uniform spherical prior on $\beta$. Specifically, we consider Bayesian linear regression with a uniform  prior on  $S^{p-1}(\sqrt{p}) \coloneqq \{ \beta \in \mathbb{R}^p: \lVert \beta \rVert = \sqrt{p} \}$. Crucially, we assume throughout that we are in the \emph{well-specified} setting, i.e. the data $\{(y_i,x_i): 1\leq i \leq n\}$ is also generated from the model $y_i = x_i^{\top}\beta_0 + \varepsilon_i$, where $\beta_0 \sim \pi_0$, and $\pi_0$ henceforth denotes the uniform distribution on $S^{p-1}(\sqrt{p})$.

We assume that $x_i \sim \textit{N}_p(0,I/n)$ are iid. We will work under the \emph{proportional asymptotics} setting, i.e., the number of samples $n$ and the number of features $p$ grow proportionally, with $n/p\to \alpha \in (0, \infty)$. For convenience, we will track the asymptotics in terms of $p$, and will think of $n$ as a function of $p$. 
Note that in this setting, the \textit{partition function} is 
\begin{equation}
\label{definition_of_Z_p}
    \mathcal{Z}_p = \int_{S^{p-1}(\sqrt{p})} e^{-\frac{1}{2\Delta} \lVert y - X\beta\rVert^2} \mathrm{d} \pi_0(\beta)\text{.}
\end{equation}
The partition function  $\mathcal{Z}_p$ is a function of both $y$ and $X$, so we use the notation $\mathcal{Z}_p(y,X)$ when we want to highlight its dependence on $y$ and $X$.
We refer to 
\begin{equation}
    F_p(y, X) = \frac{1}{p} \ln \mathcal{Z}_p(y,X) \nonumber 
\end{equation}
as the \textit{free energy} and
\begin{equation}
\label{definition_of_hat_H}
    \mathcal{H}(\beta) = \frac{1}{2\Delta} \lVert y - X\beta\rVert^2
\end{equation}
as the \textit{Hamiltonian}, borrowing terminology from statistical physics. Armed with these notions, we can state our main result.


\begin{theorem}
\label{main_theorem}
    Fix $\alpha \in (0,\infty)$. There exists $\Delta_0>0$ such that for all $\Delta>\Delta_0$,  as $p \to \infty$,
    \begin{equation}
        \left |  F_p - \sup_{a \in \mathbb{R}^p, \lVert a \rVert \le \sqrt{p}} f_{\text{TAP}}(a) \right | \overset{\text{P}}{\longrightarrow}   0 \text{,} \nonumber 
    \end{equation}
    where
    \begin{equation}
        f_{\text{TAP}}(a) \coloneqq -\frac{1}{2\Delta p} \lVert y - X a \rVert^2 -\frac{\alpha}{2} \ln \left (1 + \frac{1-\lVert a\rVert^2 / p}{\Delta \alpha} \right ) + \frac{1}{2} \ln \left ( 1 - \frac{\lVert a \rVert^2}{ p }\right ) \text{.} \label{eq:tap_functional} 
    \end{equation}
\end{theorem}

\begin{remark}
Note from \eqref{definition_of_Z_p} that $ 1 / (2\Delta)$ is similar to an inverse temperature parameter in this model. Thus our TAP representation is valid \emph{at high temperature}. The high-temperature requirement comes up when we lower bound the free energy in terms of the TAP formula. We emphasize that the matching upper bound is valid in general, and does not require any additional conditions. Establishing a valid TAP representation at any $\Delta>0$ remains an open problem in this setting. 
\end{remark} 

\begin{remark}
A close look at our proof indicates that the TAP representation is valid for any $\Delta>0$ if $\alpha$ is sufficiently large. In this case, the Signal-to-Noise Ratio is high, and the model is again effectively at high-temperature. We state all our results at fixed $\alpha$, with $\Delta>0$ sufficiently large, but note that all results transfer immediately to the \emph{fixed $\Delta$ large $\alpha$} regime. 
\end{remark}

\subsection{Background and Challenges}
We describe here the broader context of our result. We also take this opportunity to highlight the main differences from the existing literature, and describe the key challenges in our setting.

The TAP formalism was originally introduced for the Sherrington-Kirkpatrick model by Thouless-Anderson-Palmer in \cite{thouless1977solution}, predating Parisi's groundbreaking \emph{replica-symmetry breaking} solution. The TAP approach yields a variational representation for the free energy, as well as a fixed point equation for the global magnetizations, referred to as the TAP equations. The TAP equations were established rigorously using a number of different approaches---Stein's method \cite{chatterjee2010spin}, cavity method \cite{talagrand2010mean}, dynamical methods \cite{adhikari2021dynamical} etc. At high-temperature, the Gibbs measure is in one \emph{pure-state}, and the magnetizations satisfy the TAP equations. The behavior at low temperatures is significantly more complicated---the Gibbs measure decomposes into exponentially many \emph{pure states}, and a version of the TAP equations are expected to hold within each pure state. A rigorous version of this conjecture was verified in \cite{auffinger2019thouless} using the approximate ultrametric decomposition from \cite{jagannath2017approximate}. 

The progress on the TAP formula for the free energy is more recent. \cite{chen2018tap} verified the TAP representation by evoking the Parisi formula. More recently, direct proofs of the TAP representation have been derived in the literature. In this regard, \cite{belius2019tap} utilizes Random Matrix Theory to establish the variational formula for the spherical SK model, while subsequent developments establish these results for mixed p-spin models \cite{chen2018generalized,chen2021generalized}. In this regard, the notion of \emph{multi-samplability} plays a crucial role. We will also crucially exploit a related idea in our analysis. We emphasize this connection in our discussion of the proof strategy below. We note that in sharp contrast with the model studied in this paper, the models studied in the prior literature are \emph{null} models, i.e. do not have a latent signal. The planted setting is arguably more natural in the context of Statistics and Machine Learning. To the best of our knowledge, this is the very first proof of a TAP formula for the free energy for a planted spin glass model. We note that the random hamiltonians in the study of the Sherrington-Kirkpatrick and mixed p-spin models are centered Gaussian processes---this directly facilitates the use of tools tailored for Gaussian variables e.g. interpolation  and comparison techniques \cite{chen2018generalized,chen2021generalized}. This is no longer true in our setting, and thus necessitates new ideas. In this setting, the hamiltonian roughly looks like a two-spin model, with the sample covariance matrix as the couplings among the spins. We directly utilize this structure in our proofs, exploiting spectral properties of this matrix. We expect our proof ideas to be applicable for more general random matrix ensembles, e.g., the orthogonally invariant matrix ensembles \cite{parisi1995mean,barbier2018mutual}.

The behavior of the TAP free energy has been examined in the setting of finite rank spiked matrix models in \cite{ghorbani2019instability}. \cite{ghorbani2019instability} establishes that the Naive Mean Field approximation exhibits an undesirable instability property below the Information Theoretic threshold, thus rendering variational methods based on this approximation untrustworthy in practice. In turn, it suggests using the TAP framework to carry out variational inference in this setting. In follow up work \cite{fan2021tap,celentano2021local}, the accuracy of TAP-approximation based variational inference has been established in the context of the $\mathbb{Z}_2$ synchronization problem, a special case of the general setting examined in \cite{ghorbani2019instability}.

\subsection{Proof strategy}

Our proof has two separate arguments---first, we lower bound the log-partition function $\log \mathcal{Z}_p$ in terms of the TAP free energy $f_{\text{TAP}}(a)$, as defined in \eqref{eq:tap_functional}. To this end, fix $a \in \mathbb{R}^p$, $\|a\| \leq \sqrt{p}$. We re-center the hamiltonian $\mathcal{H}(\beta)$ around the vector $a$; this yields the term $- \|y - Xa\|^2 / (2\Delta p)$ in $f_{\text{TAP}}(a)$. To obtain a valid lower bound, we restrict the integral in $\mathcal{Z}_p$ to the intersection of $S^{p-1}(\sqrt{p})$ and a subspace that is orthogonal to $a$ and the resulting external magnetization $X^{\top}(y-Xa)$ from the previous re-centering step. The volume of this \emph{band} is $ [ \ln (1- \|a\|^2/p) ] / 2 $---this contributes the third term in $f_{\text{TAP}}(a)$. Finally, the model restricted to the \emph{band} is a spherical two-spin model with hamiltonian $H(\beta) = - \|X(\beta-a)\|^2 / (2 \Delta)$. The free energy of this model is calculated using the second moment method, and contributes  $-(\alpha / 2) \ln [1 + ( 1- \|a\|^2/p) / (\Delta \alpha ) ] $ in $f_{\text{TAP}}(a)$. This yields a natural interpretation of the Onsager correction term in this problem. We note that this step explicitly uses the ``high-temperature" condition stated in the Theorem statement. This shows that $( \ln \mathcal{Z}_p ) / p \geq f_{\text{TAP}}(a)$ \emph{uniformly} over all $a \in \mathbb{R}^p$ such that $\|a\|/\sqrt{p}$ is bounded away from one. We argue separately that the end-point can be neglected. Our argument is similar to the one of \cite{belius2019tap} in the context of the spherical two-spin Sherrington-Kirkpatrick model. 

Our main contribution is the corresponding upper bound on the free energy in terms of the TAP free energy functional. To this end, we crucially leverage \emph{overlap concentration} properties under the well-specified posterior distributions \cite{barbier2019adaptive}. Roughly, this implies that given a vanishing amount of additional side information, with high-probability (over the law of the disorder $(y,X)$), we can sample many iid \emph{replicas} $\beta^1, \cdots, \beta^{m_p}$ from the posterior distribution such that all the pairwise overlaps are concentrated around a fixed value. Such overlap values are also termed as \emph{multi-samplable} in \cite{subag2018free} in the study of the TAP formalism for mixed p-spin spherical spin glasses. This overlap concentration property allows us to derive an upper bound to the log-partition function $\ln \mathcal{Z}_p$ in terms of a restricted partition function for a \emph{replicated} system, where all the pairwise overlaps are restricted to approximately the same value. We now set $a$ as the sample average of these replicas, and re-center the hamiltonian around this vector. The concentration of all pairwise overlaps around a fixed value implies that the $\beta^{i}-a$'s are approximately orthogonal, and this allows us to cancel the external magnetic field $X^{\top}(y-Xa)$ created from the re-centering step. Finally, due to the approximate orthogonality of the $\beta^i-a$ vectors, the integral in the replicated system approximately de-couples, and provides $m_p$ copies of the Onsager correction and volume contributions. In particular, this establishes that 
$ ( \ln \mathcal{Z}_p ) / p \leq f_{\text{TAP}}(a) + o(1)$ for some (random) $a \in \mathbb{R}^p$, and thus yields the desired upper bound.

We note that the overlap concentration property plays a very crucial role in our proof. This concentration property has been established recently for well-specified Bayesian posteriors (see e.g. \cite{barbier2020strong}). It should be possible to combine the ideas in this paper with overlap concentration to establish similar TAP formulae in other Bayesian inference problems. We leave this for future work.

\subsection{Directions for future research} 
We collect here some open questions for future research. 

\begin{itemize}
    \item[(i)] An immediate open question is to extend our TAP representation to any $\Delta>0$. We note that the high-temperature condition is explicitly used to calculate the free energy on the band---this is the only step in our proof which uses this condition. 

    \item[(ii)] The main question for future research concerns a proof of Conjecture \ref{conj:TAP_reg}. We believe that the ideas introduced in this paper will be crucial in this effort. The spherical prior crucially allows us to use ideas from random matrix theory; specifically, after the re-centering step described above, one has to compute the free energy of a two spin model on the sphere. Further, one requires this estimate to be \emph{uniform} over the choice of the center. Exploiting the spherical symmetry, we use asymptotics of spherical integrals \cite{guionnet2005fourier} to derive these uniform estimates. This symmetry is crucially absent for other natural priors, e.g., product priors. Thus our arguments do not generalize directly to these settings.

    \item[(iii)] In Statistics and Machine Learning, one is often interested in Variational Inference based on the TAP free energy. Specifically, given Theorem \ref{main_theorem}, a natural recipe is to start at some appropriate initialization, and subsequently optimize $f_{\text{TAP}}$ directly by some out-of-the-box optimization algorithm such as gradient descent. Of course, $f_{\text{TAP}}$ is non-convex in general, and thus the algorithm will usually reach a local optimum, rather than the global maximum. This yields a natural estimator, and one naturally wishes to study the statistical performance of this estimator. It would be of interest to characterize this performance. Such results have been achieved recently in the context of the $\mathbb{Z}_2$ synchronization problem \cite{fan2021replica,celentano2021local}. We expect these insights for $\mathbb{Z}_2$ synchronization to also be useful in the context of Bayesian linear regression. 
\end{itemize}

\noindent
\textbf{Notations:} 
We use the usual Bachmann-Landau notation $O(\cdot)$, $o(\cdot)$, $\Theta(\cdot)$ for sequences. For a sequence of random variables $\{X_p : p \geq 1\}$, we say that $X_p = o(1)$ if $X_p \stackrel{P}{\to} 0$ as $p \to \infty$.

We use $\mathbb{P}[\cdot]$ and $\mathbb{E}[\cdot]$ to denote the probability and expectation under the joint distribution of the data $(y, X)$. For other measures, we add a suitable subscript to the probability and expectation operators to emphasize the distribution under consideration. 
In our computations, we will often restrict the calculations to a \emph{good} set. We collect this set here. Let $\kappa = 100$. Define a set
\begin{equation}
\label{eq:def_Omega}
    \Omega \coloneqq \left  \{ (X, \epsilon, \beta_0): \sigma_{\text{max}}(X) < 1 + \alpha^{-1/2} + \kappa, \lVert \varepsilon \rVert < \kappa  \sqrt{p \Delta} \right  \} \text{,}
\end{equation}
where $\sigma_{\max}(X)$ denotes the largest singular value of $X$. 
By Lemma \ref{large_deviation_characterization_for_the_largest_eigenvalue_of_Weshart_matrices}, one has
\begin{equation}
    \lim_{p \to \infty} \mathbb{P}((X, \epsilon, \beta_0) \in \Omega)  =  0 \text{.} \nonumber 
\end{equation}
\noindent 
Note that with a slight abuse of notations, we will use $(X,y) \in \Omega$ and $(X, \epsilon, \beta_0) \in \Omega$ exchangeably. Throughout this paper, unless otherwise specified, a sequence of events is said to occur \emph{with high probability} if it has probability $1-o(1)$ under the joint distribution $(X,y)$.

Throughout, we use $C,c, C_1, c_1, C_2,c_2 \cdots $ to denote positive constants independent of $n,p$, possibly depending on the parameters $\Delta$ and $\alpha$. Further, these constants can change from line to line. For any square symmetric matrix $A$, $\|A\|_2$ and $\|A\|_F$ denote the matrix operator norm and the Frobenius norm respectively. Finally, $\langle \cdot \rangle$ will denote the expectation operator under the posterior distribution \eqref{eq:posterior}.

\noindent
\textbf{Outline:} The rest of the article is structured as follows: We derive a lower bound on the log-partition function via the TAP free energy in Section \ref{sec:lower_bound}, while Section \ref{sec:upper_bound} establishes a matching upper bound. Section \ref{sec:combination} combines these two bounds, and establishes Theorem \ref{main_theorem}. We crucially utilize the notion of overlap concentration in our upper bound proof. The proof of overlap concentration is a direct adaptation of existing arguments in the literature. For the sake of completeness, we include a proof in our setting in Appendix \ref{sec:overlap_conc}. 


\section{Lower bound on the free energy}
\label{sec:lower_bound} 

To get a lower bound of the free energy $F_p = \frac{1}{p} \ln \mathcal{Z}_p$, we restrict the integral in (\ref{definition_of_Z_p}) to specific subsets of $S^{p-1}(\sqrt{p})$, which yields a natural lower bound. Following this strategy, we will first state Theorem \ref{band_theorem} in Section \ref{sec:lower_bound_proof_from_band_theorem}; this specifies the subsets that we look at. We then prove Theorem \ref{lower_bound} assuming Theorem \ref{band_theorem}. The proof of Theorem \ref{band_theorem} itself will be deferred to Section \ref{sec:proof_of_band_theorem}. Finally, in Section \ref{sec:lemmas_for_proof_of_lower_bound}, we provide proofs of some intermediate lemmas that are used to prove Theorem \ref{band_theorem}.

\begin{theorem}[Lower bound]
\label{lower_bound}
Fix $\alpha \in (0,\infty)$. There exists $\Delta_0>0$ such that for all $\Delta> \Delta_0$, 
    any $\eta >0$, $c \in (0,1)$,
    \begin{equation}
        \mathbb{P} \Big (  \frac{1}{p} \ln \mathcal{Z}_p \le \sup_{a \in \mathbb{R}^p,  \lVert a \rVert \le (1 - c) \sqrt{p}} f_{\text{TAP}}(a) - \eta \Big) \to 0 \text{.} \nonumber 
    \end{equation}
\end{theorem}

\subsection{Proof of Theorem \ref{lower_bound} using Theorem \ref{band_theorem}}
\label{sec:lower_bound_proof_from_band_theorem}

\begin{theorem}
\label{band_theorem}
    Let $\{v_1, v_2\}$ be an orthogonal basis of length $\sqrt{p}$ of the 2-dim linear subspace  $\mathrm{Span}(a, X^{\top} (y - X a)) \subset \mathbb{R}^p$. For small positive $\epsilon$ and $a \in \mathbb{R}^p$, define
    \begin{equation}
    \label{eq:def_B}
        \operatorname{B}(a,\epsilon) = \{ \beta: v_i^{\top} (\beta-a)  \in (- p \epsilon, p \epsilon), \, \forall i \in \{1,2\} \} \text{.} 
    \end{equation}
    There exists $\Delta_0>0$ such that for all $\Delta>\Delta_0$, for any $c\in[0,1]$, there exists $C>0$ such that as $p \to \infty$, 
    \begin{equation}
    \label{eq:band_theorem}
        \begin{aligned}
        \mathbb{P} \Bigg ( \sup_{\lVert a \rVert \le (1-c)\sqrt{p}} \Bigg | \frac{1}{p} & \ln \Big ( \int_{\operatorname{B}(a,\epsilon) } e^{-\frac{1}{2 \Delta} (\beta- a )^{\top} X^{\top} X (\beta - a) } \mathrm{d} \pi_0(\beta) \Big )\\
        - \Big [ -\frac{\alpha}{2} & \ln \Big ( 1 + \frac{(1-\lVert a\rVert^2 / p)}{\Delta \alpha} \Big ) + \frac{1}{p} \ln \operatorname{Vol}(a, \epsilon)  \Big ] \Bigg  | \le C \epsilon \Bigg ) \to 1, \nonumber 
        \end{aligned}
    \end{equation} 
where $\operatorname{Vol}(a, \epsilon) \coloneqq \operatorname{Vol}(\operatorname{B}(a,\epsilon) ) = \mathbb{E}_{\beta \sim \pi_0}[\mathbf{1}_{\beta \in \operatorname{B}(a,\epsilon)}]$. 
\end{theorem}

\begin{proof}[Proof of Theorem \ref{lower_bound}]
Note that by the definition of $\Omega$, i.e. (\ref{eq:def_Omega}), for any $(X,y) \in \Omega$,
\begin{align}
 &  \lVert X^{\top} (y - Xa)\rVert^2 \le \sigma^2_{\text{max}}(X) \lVert X(\beta_0 - a) + \varepsilon \rVert^2  \nonumber \\
   & \le \sigma^2_{\text{max}}(X) \left [ 4 p \sigma^2_{\text{max}}(X) + \lVert \varepsilon \rVert^2 + 4 \sqrt{p} \sigma_{\text{max}}(X) \lVert \varepsilon \rVert \right ] = O(p) \text{.} \label{eq:X_transpose_y_minus_Xa_order_p}
\end{align}
From now on, we only consider $(X,y) \in \Omega$. Since the integrand in (\ref{definition_of_Z_p}) is always strictly positive and $\operatorname{B}(a,\epsilon)$ is a subset of $S^{p-1}(\sqrt{p})$, naturally we have the following lower bound
\begin{equation}
    \mathcal{Z}_p = \int_{S^{p-1}(\sqrt{p})} e^{-\frac{1}{2\Delta} \lVert y - X\beta\rVert^2}  \mathrm{d} \pi_0(\beta) \ge \int_{\operatorname{B}(a,\epsilon)} e^{-\frac{1}{2\Delta} \lVert y - X\beta\rVert^2} \mathrm{d} \pi_0(\beta) \text{.} \nonumber 
\end{equation}
Upon recentering the exponent around $a$, noting that $X^{\top}(y - X a) \in \operatorname{Span}(v_1,v_2)$, where $\operatorname{Span}(v_1, v_2)$ denotes the linear subspace of $\mathbb{R}^p$ generated by $v_1$ and $v_2$, by the definition of $\operatorname{B}(a,\epsilon)$, together with (\ref{eq:X_transpose_y_minus_Xa_order_p}), we have
\begin{equation}
    \begin{aligned}
        \lVert y - X\beta\rVert^2 &= \lVert y - X a\rVert^2 + \lVert X(\beta - a)\rVert^2 - 2 ( X^{\top} ( y  - X a ))^{\top}( \beta - a ) \\
        & = \lVert y - X a\rVert^2 + \lVert X(\beta - a)\rVert^2 + 2  \epsilon O(p)
    \end{aligned} \nonumber 
\end{equation}
for any $\beta \in \operatorname{B}(a,\epsilon)$. Thus we have
\begin{equation}
    \begin{aligned}
        \mathcal{Z}_p 
        &\ge e^{-\frac{1}{2 \Delta} \lVert y - X a\rVert^2 + O(p\epsilon)} \int_{\operatorname{B}(a,\epsilon) } e^{-\frac{1}{2 \Delta} \lVert X(\beta - a)\rVert^2} \mathrm{d} \pi_0(\beta) \text{.}
    \end{aligned} \nonumber 
\end{equation}
Equivalently, there is a positive constant $C'$, such that
\begin{equation}
    \ln \mathcal{Z}_p \ge -\frac{1}{2 \Delta} \lVert y - X a\rVert^2 - C' \epsilon + \ln \left ( \int_{\operatorname{B}(a,\epsilon)} e^{-\frac{1}{2 \Delta} \lVert X(\beta - a)\rVert^2} \mathrm{d} \pi_0(\beta) \right ) \text{.} \nonumber 
\end{equation}
    Note that for any 2-dimensional linear subspace $\mathcal{L} \subset \mathbb{R}^{p}$, length of the projection of $\beta$ on to $\mathcal{L}$ divided by $\sqrt{p}$ has density
    \begin{equation}
        f_{\text{Proj}}(x)  = \frac{1}{\sqrt{\pi}} \frac{\Gamma\left(\frac{p}{2}\right)}{\Gamma\left(\frac{p-2}{2}\right)}\left(1-x^{2}\right)^{\frac{p-4}{2}} \nonumber 
    \end{equation}
    (see for instance of \cite[(2.8)]{belius2019tap}). Without loss of generality, we only consider the non-degenerate case, namely when $a$ and $X^{\top}( y - Xa)$ are not linearly dependent.
    Thus when $\epsilon$ is small enough, 
    \begin{equation}
        \frac{1}{p} \ln \operatorname{Vol}(a,\epsilon) \ge \frac{1}{p} \ln \left [ \sqrt{2}  \epsilon \cdot f_{\beta^{\perp}}\Big( \frac{\lVert a \rVert}{\sqrt{p}} -  \frac{\epsilon}{\sqrt{2}} \Big) \cdot \min \left \{ \frac{\epsilon}{\sqrt{2} \pi \lVert a \rVert} , 1 \right \} \right ] = \frac{1}{2} \ln \left ( 1 - \frac{\lVert a \rVert^2}{ p }\right ) + o(1) \text{.} \nonumber 
    \end{equation}
    Note that the $o(1)$ term is uniform over different choices of $a$ such that $\lVert a \rVert \le (1 - c) \sqrt{p}$. By Theorem \ref{band_theorem}, denoting the event on the LHS of (\ref{eq:band_theorem}) as $\Omega_p$, 
    \begin{equation}
        \begin{aligned}
            & \mathbb{P} \left ( \frac{1}{p} \ln \mathcal{Z}_p \ge \sup_{a \in \mathbb{R}^p, \lVert a \rVert \le (1-c) \sqrt{p}} f_{\text{TAP}}(a) - 2 (C + C') \epsilon + o(1)\right ) \\
            \ge & \mathbb{P} \left ( \left \{\frac{1}{p} \ln \mathcal{Z}_p \ge \sup_{a \in \mathbb{R}^p, \lVert a \rVert \le (1-c) \sqrt{p}} f_{\text{TAP}}(a) - 2 (C + C') \epsilon + o(1) \right  \}  \cap \Omega \right ) \\ 
            \ge & \mathbb{P} \Bigg (  \Bigg \{\frac{1}{p} \ln \mathcal{Z}_p \ge -\frac{1}{2 \Delta} \lVert y - X a\rVert^2 - C' \epsilon + \ln \left ( \int_{\operatorname{B}(a,\epsilon)} e^{-\frac{1}{2 \Delta} \lVert X(\beta - a)\rVert^2} \mathrm{d} \pi_0(\beta)   \right), \ \\
            &  \forall a \in \mathbb{R}^p: \lVert a \rVert \le (1 - c)\sqrt{p} \Bigg \}  \cap \Omega \cap \Omega_p  \Bigg ) \\ 
            = & \mathbb{P} ( \Omega \cap \Omega_p) \to 1 
            \text{,}
        \end{aligned} \nonumber 
    \end{equation}
    where $C$ is the one in statement of Theorem \ref{band_theorem}. In other words, since $\epsilon$ can be chosen to be arbitrarily small, we have for any $\eta > 0$,
    \begin{equation}
        \lim_{p \to \infty} \mathbb{P} \left ( \frac{1}{p} \ln \mathcal{Z}_p \le   \sup_{a \in \mathbb{R}^p, \lVert a \rVert \le (1-c) \sqrt{p}} f_{\text{TAP}}(a) - \eta \right ) = 0. \nonumber 
    \end{equation}
\end{proof}

\subsection{Proof of Theorem \ref{band_theorem}}
\label{sec:proof_of_band_theorem}

Building towards the proof of Theorem \ref{band_theorem}, we first introduce the following lemma, which gives an explicit formula for the high-dimensional limit of the free energy in a system without external field. We establish this using a second moment method argument 
(see Section \ref{sec:lemmas_for_proof_of_lower_bound} for 
its proof). Note that we use the notation $\beta^{\top} X^{\top} X \beta $ instead of $\lVert X \beta \rVert^2$ in this subsection since we would like to emphasize the fact that our proofs depend  strongly on the properties of the sample covariance matrix $X^{\top} X$.

\begin{lemma}[High-dimensional limit of free energy: a model without external field]
\label{limit_of_free_energy_without_external_field}
    Fix $\alpha \in (0,\infty)$. There exists $\Delta_0>0$ such that for all $\Delta > \Delta_0$,  as $p \to \infty$
    \begin{equation}
          \frac{1}{p} \ln \left ( \int_{S^{p-1}(\sqrt{p})} e^{-\frac{1}{2 \Delta}\beta^{\top} X^{\top} X \beta } \mathrm{d} \pi_0(\beta) \right ) \overset{\text{P}}{\longrightarrow} -\frac{\alpha}{2} \ln \left ( 1 + \frac{1}{\Delta \alpha} \right ) \text{.} \nonumber 
    \end{equation} 
\end{lemma}
\noindent 
Note the left hand side is Lipschitz in $T = \frac{1}{\Delta}$ with high probability. Thus for any $\Delta_0>0$, we can divide the interval $[0,\frac{1}{\Delta_0}]$ into a sufficiently fine grid. Lemma \ref{limit_of_free_energy_without_external_field} holds simultaneously for  each grid-point; combining this with the Lipschitz continuity of the free energy in $\frac{1}{\Delta}$, we immediately obtain the following uniform convergence statement. 
\begin{lemma}
\label{limit_of_free_energy_without_external_field_uniform}
Fix $\alpha \in (0,\infty)$. There exists $\Delta_0>0$ such that  as $p \to \infty$
\begin{equation}
    \sup_{\Delta \ge \Delta_0} \left | \frac{1}{p} \ln \left ( \int_{S^{p-1}(\sqrt{p})} e^{-\frac{1}{2 \Delta}\beta^{\top} X^{\top} X \beta } \mathrm{d} \pi_0(\beta) \right )   - \left ( - \frac{\alpha}{2} \ln \left ( 1 + \frac{1}{\Delta \alpha} \right ) \right ) \right | \overset{\text{P}}{\longrightarrow}  0 \text{.} \nonumber 
\end{equation}
\end{lemma}
As the definition of $\operatorname{B}(a,\epsilon)$ in \eqref{eq:def_B} suggests, we want to consider integral over the intersection of the sphere and a $p-2$ dimensional linear space, which can be formalized into the following uniform convergence statement. 
\begin{lemma}[High-dimensional limit of free energy without external field: on a linear subspace]
\label{limit_of_free_energy_without_external_field_on_a_linear_subspace}
    Fix $\alpha \in (0,\infty)$. There exists $\Delta_0>0$ such that for all $\Delta>\Delta_0$, as $p \to \infty$
    \begin{equation}
    \sup_{\Delta \ge \Delta_0,\text{  } u, v \in \mathbb{R}^p } \left | \frac{1}{p} \ln \left ( \int_{S^{p-1}(\sqrt{p}) \cap \operatorname{Span}(u, v)^{\perp}} e^{-\frac{1}{2 \Delta}\beta^{\top} X^{\top} X \beta } \mathrm{d} \pi^{\operatorname{Span}(u, v)^{\perp}}(\beta) \right ) - \left (-\frac{\alpha}{2} \ln \left ( 1 + \frac{1}{\Delta \alpha} \right) \right ) \right | \overset{\text{P}}{\longrightarrow} 0 \text{,} \nonumber 
    \end{equation}
    where $\operatorname{Span}(u, v)$ denotes the linear subspace of $\mathbb{R}^p$ generated by $u$ and $v$. 
    Note that $\pi^{\operatorname{Span}(u, v)^{\perp}}$ refers to the uniform measure on $S^{p-1}(\sqrt{p}) \cap \operatorname{Span}(u, v)^{\perp}$.
\end{lemma}

Now we proceed to prove Theorem \ref{band_theorem}, assuming the convergence results mentioned above. The proofs of these results are included in Section \ref{sec:lemmas_for_proof_of_lower_bound}.

\begin{proof}[Proof of Theorem \ref{band_theorem}]
    Recall the definition of $\Omega$ in (\ref{eq:def_Omega}). From now on, we only consider $(X,y) \in \Omega$ for the rest of this proof. Fix any $a \in \mathbb{R}^{p}$ with $ \lVert a \rVert \le (1 - c) \sqrt{p}$. Let $\hat{\beta} = \beta - a$. Define
    \begin{equation}
    \label{def:hamiltonian_without_external_field}
        H(\hat{\beta}) \coloneqq \frac{1}{2 \Delta} \hat{\beta}^{\top} X^{\top} X \hat{\beta} = \frac{1}{2 \Delta} (\beta- a )^{\top} X^{\top} X (\beta - a)  \text{.}
    \end{equation}
    Note that $H$ is different from the Hamiltonian $\mathcal{H}$  defined in (\ref{definition_of_hat_H}). By (\ref{eq:X_transpose_y_minus_Xa_order_p}), $\lVert X^{\top} (y-Xa) \rVert = O(\sqrt{p})$. By the definition of $\operatorname{B}(a,\epsilon)$ in (\ref{eq:def_B}), for any $\beta \in \operatorname{B}(a,\epsilon)$, one has 
    \begin{equation}
       | a^{\top} \hat{\beta} | = | a^{\top}( \beta - a)| < c' p\epsilon \nonumber 
    \end{equation}
    and
    \begin{equation}
        | (X^{\top}(y - X a))^{\top} \hat{\beta} \rangle | = | (y - X a)^{\top} X(\beta - a)  |  < c' p\epsilon \text{,} \nonumber 
    \end{equation}
    for some positive constant $c'$. Therefore
    \begin{equation}
        \lVert \hat{\beta}\rVert^2 = \lVert \beta \rVert^2 - \lVert a\rVert^2 - 2 (\beta - a)^{\top} a = p - \lVert a\rVert^2 + O(p\epsilon) \nonumber 
    \end{equation}
    Let $\hat{\beta}^{\perp}$ be the projection of $\hat{\beta}$ onto the hyper-plane $\operatorname{Span}(v_1, v_2)^{\perp}$. For the sake of notation clarity, let $ \beta^{\perp} = \sqrt{p} \hat{\beta}^{\perp} / \lVert \hat{\beta}^{\perp} \rVert $ and $\gamma = \lVert \hat{\beta}^{\perp} \rVert / \sqrt{p}$. 
    Note that, again by the definition of $\operatorname{B}(a,\epsilon)$
    \begin{equation}
        \lVert \hat{\beta} - \gamma \beta^{\perp} \rVert = \lVert \hat{\beta} - \hat{\beta}^{\perp} \rVert < c' \sqrt{p} \epsilon \text{,} \nonumber 
    \end{equation}
    and
    \begin{equation}
        \gamma = \gamma_0 + O(\epsilon) \text{,} \nonumber 
    \end{equation}
    where $\gamma_0 \coloneqq \sqrt{ p - \lVert a \rVert^2} / \sqrt{p}$. Therefore, 
    \begin{equation}
        H(\hat{\beta}) = H(\gamma \beta^{\perp}) + O(p\epsilon) = H(\gamma_0 \beta^{\perp}) + O(p\epsilon) \text{.} \nonumber 
    \end{equation}
    which is due to the fact that $ \lVert \nabla H(\beta) \rVert =  \lVert - X^{\top} X \beta  / \Delta \rVert  \le (\sqrt{p} / \Delta) \lVert X^{\top} X \rVert_2  = O(\sqrt{p})$ for $(X,y) \in \Omega$. Now we come back to the targeted integral
    \begin{equation}
        \int_{\operatorname{B}(a,\epsilon)} e^{-\frac{1}{2 \Delta} (\beta- a )^{\top} X^{\top} X (\beta - a) } \mathrm{d} \pi_0(\beta) =  \mathbb{E}_{\pi_0}[\mathbf{1}_{\operatorname{B}(a,\epsilon) } e^{-H(\gamma_0 \beta^{\perp}) + O(p\epsilon)}] \text{.} \nonumber 
    \end{equation}
    Since $\beta^{\perp}$ is uniformly distributed on the unit sphere intersected with $\operatorname{Span}(a, X^{\top} (y - X a))^{\perp}$ under $\pi_0$ and $\gamma$ is independent of $\mathbf{1}_{\beta \in \operatorname{B}(a, \epsilon)}$, one has 
    \begin{equation}
    \begin{aligned}
        & \mathbb{E}_{\pi_0} [\mathbf{1}_{\operatorname{B}(a,\epsilon) } e^{-H(\gamma_0 \beta^{\perp}) + O(p\epsilon)}]
        \\
        =& e^{O(p\epsilon)} \mathbb{E}_{\pi_0}[\mathbf{1}_{\operatorname{B}(a,\epsilon) }] \int_{S^{p-1}(\sqrt{p - \lVert a\rVert^2 }) \cap \operatorname{Span}(v_1, v_2)^{\perp}} e^{-\frac{1}{2 \Delta}\beta^{\top} X^{\top} X \beta } \mathrm{d} \pi^{\operatorname{Span}(v_1, v_2)^{\perp}}(\beta)\\
        =& e^{O(p\epsilon)} \operatorname{Vol}(a, \epsilon) \int_{S^{p-1}(\sqrt{p}) \cap \operatorname{Span}(v_1, v_2)^{\perp}} e^{-\frac{1}{2 \tilde{\Delta}}\beta^{\top} X^{\top} X \beta } \mathrm{d} \pi^{\operatorname{Span}(v_1, v_2)^{\perp}}(\beta) \\
        =&  e^{O(p\epsilon)} \operatorname{Vol}(a, \epsilon) \mathbb{E}^{\operatorname{Span}(v_1, v_2)^{\perp}} \left [  \exp \left ( - \frac{1}{2 \tilde{\Delta}} \beta^{\top} X^{\top} X \beta   \right )  \right ] \text{,}
    \end{aligned}  \nonumber   
    \end{equation}
    where 
    $$
    \tilde{\Delta} = \frac{p}{p - \lVert a\rVert^2}\Delta \text{.}
    $$
    Note that for any $a$ such that 
    $\lVert a \rVert \le (1 - c) \sqrt{p}$, we have $\tilde{\Delta} > \Delta/2$. By Lemma \ref{limit_of_free_energy_without_external_field_on_a_linear_subspace} and the fact that $\mathbb{P}(\Omega) \to 1$ as $p \to \infty$, we finally get with high probability
    \begin{align}
        &\frac{1}{p} \ln \int_{\operatorname{B}(a,\epsilon)} e^{-\frac{1}{2 \Delta} (\beta- a )^{\top} X^{\top} X (\beta - a) } d \pi_0(\beta)  \nonumber \\
        =& O(\epsilon) + \frac{1}{p} \ln \operatorname{Vol}(a, \epsilon) -\frac{\alpha}{2} \ln (1 + \frac{(1-\lVert a\rVert^2 / p)}{\Delta \alpha}) + o(1) \text{,} \nonumber 
    \end{align}
    in which the $o(1)$ terms are uniform with respect to the choice of $\tilde{\Delta}$, equivalently, with respect to the choice of $a$.
\end{proof}

\subsection{Concentration of free energy in the model without external field}
\label{sec:lemmas_for_proof_of_lower_bound}

In this subsection, we will prove Lemma \ref{limit_of_free_energy_without_external_field} and Lemma \ref{limit_of_free_energy_without_external_field_on_a_linear_subspace}. 

\subsubsection{Proof of Lemma \ref{limit_of_free_energy_without_external_field}}

Recall that $H(\beta) = H(X, \beta) = (1/ {2 \Delta} ) \lVert X \beta \rVert^2$, which was first defined in (\ref{def:hamiltonian_without_external_field}).
We define the partition function $Z_p(\Delta, X)$ as
\begin{equation}
    Z_p(\Delta, X) \coloneqq \int_{S^{p-1}(\sqrt{p})} e^{-\frac{1}{2 \Delta}\beta^{\top} X^{\top} X \beta } \mathrm{d} \pi_0(\beta) \text{.} \nonumber 
\end{equation}
which is different from $\mathcal{Z}_p$ defined in (\ref{definition_of_Z_p}). Let
\begin{equation}
    \Phi_p(\Delta, X) \coloneqq \frac{1}{p} \ln \int_{S^{p-1}(\sqrt{p})} e^{-\frac{1}{2 \Delta} \beta^{\top} X^{\top} X \beta } \mathrm{d} \pi_0(\beta) = \frac{1}{p} \log \mathbb{E}_{\pi_0} [\exp (H(\beta)) ] = \frac{1}{p} \ln Z_p(\Delta,X) \nonumber 
\end{equation}
be the \textit{quenched} free energy of the model without external field and
\begin{equation}
    \phi_p(\Delta) \coloneqq \frac{1}{p}  \ln  \mathbb{E} \int_{S^{p-1}(\sqrt{p})} e^{-\frac{1}{2 \Delta}\beta^{\top} X^{\top} X \beta } \mathrm{d} \pi_0(\beta) = \frac{1}{p}  \ln \mathbb{E}[ Z_p(\Delta,X)]  \nonumber 
\end{equation}
be the \textit{annealed} free energy.

To prove Lemma \ref{limit_of_free_energy_without_external_field}, we need the following concentration result about $\Phi_p(\Delta, X)$.
\begin{lemma}
\label{concentration_ineq_for_second_moment_method}
    For any $\Delta > 0$, $\alpha \in (0,\infty)$ and $\delta >0 $, there exists a positive constant $C$ and positive integer $p_0$ (depending only on $\Delta, \alpha, \delta$) such that for any $p > p_0$
    \begin{equation}
        \mathbb{P}(\left |  \Phi_p(\Delta, X) - \mathbb{E}\Phi_p(\Delta, X)   \right | > \delta ) \le e^{-Cp}. \nonumber 
    \end{equation}
\end{lemma}

\begin{proof}
    Let $\sigma_{\text{max}}(X)$ be the largest singular value of $X$. Define 
    \begin{equation}
         \mathcal{G} \coloneqq \left \{ X : \sigma_{\text{max}}(X) <  \frac{1}{\sqrt{n}} (\sqrt{p} + \sqrt{n} + t)         \right \}. \nonumber  
    \end{equation}
    First,  we prove that $\Phi_p(\Delta, X) $ is Lipschitz as a function of $X$ on $\mathcal{G}$, with Lipschitz constant $K_p = 2(1 + \alpha^{-1/2} + t/\sqrt{n}) / \sqrt{p} $. By triangle inequality, for any $X_1, X_2 \in \mathcal{G}$,
    \begin{equation}
        \begin{aligned}
            |H(X_1, \beta) - H(X_2, \beta)| &\le |\beta^{\top} (X_1 - X_2)^{\top} X_2 \beta| + |  \beta^{\top} X_1^{\top} (X_1 - X_2) \beta  | \\
            &\le \lVert \beta \rVert^2 \left [ \lVert (X_1 - X_2)^{\top} X_2 \rVert_2 + \lVert X_1^{\top} (X_1 - X_2) \rVert_2 \right ] \\
            &\le \sqrt{p}\left [ (\sigma_{\text{max}}(X_1) + \sigma_{\text{max}}(X_2))    \lVert \sqrt{p} X_1 - \sqrt{p} X_2 \rVert_2 \right ]\\
            &\le 2 \alpha^{-1/2} (\sqrt{p} + \sqrt{n} + t)\lVert \sqrt{p} X_1 - \sqrt{p} X_2 \rVert_F \text{,}
        \end{aligned} \nonumber 
    \end{equation}
    which implies
    \begin{equation}
        \begin{aligned}
            | \Phi_p(\Delta, X_1) - \Phi_p(\Delta, X_2) | &= \frac{1}{p} \left | \ln \frac{\int_{S^{p-1}(\sqrt{p})} e^{- H(X_1,\beta)} d\pi_0(\beta)}{\int_{S^{p-1}(\sqrt{p})} e^{- H(X_2,\beta )} d\pi_0(\beta)} \right | \\
            &\le \frac{1}{p} \left [ 2 \alpha^{-1/2} (\sqrt{p} + \sqrt{n} + t)\lVert \sqrt{p} X_1 - \sqrt{p} X_2 \rVert_F \right ] \\
            &= \frac{2\alpha^{-1/2} (1 + \sqrt{\alpha} + t/\sqrt{p})}{\sqrt{p}}  \lVert \sqrt{p} X_1 - \sqrt{p} X_2 \rVert_F \text{.}
        \end{aligned} \nonumber 
    \end{equation}
    On the other hand, by Lemma \ref{large_deviation_characterization_for_the_largest_eigenvalue_of_Weshart_matrices},
        $\mathbb{P} (\mathcal{G}^c) \le 2 e^{ - t^2}$. 
    Note that $\Phi(\Delta, 0) = 0$ and 
    \begin{equation}
        \mathbb{E} \left [ \Phi_p^2(\Delta, X)\right ] \le \mathbb{E} \left [ \frac{1}{2 \Delta}  \sigma^2_{\text{max}}(X)\right ] < \infty \text{.} \nonumber 
    \end{equation}
    Therefore, by Lemma \ref{concentration_of_almost_Lipshitz_functions}, there exists a positive constant $C$ such that for any $r > 6 (C + \sqrt{np K_p})  \sqrt{\mathbb{P}(\mathcal{G}^c)}$ we have
    \begin{equation}
        \mathbb{P} (\left |  \Phi_p(\Delta, X) - \mathbb{E}\Phi_p(\Delta, X)   \right | > r ) \le 2 e^{-\frac{r^2}{16 K_p^2}} + \mathbb{P}(\mathcal{G}^c) \text{.} \nonumber 
    \end{equation}
    Now let $\delta>0$ be any fixed constant. Let $t = \sqrt{p}$. Then there exists $ p_0 \in \mathbb{N}$ such that for any $p > p_0$, $\delta > 6 (C + \sqrt{np K_p})  \sqrt{\mathbb{P}(\mathcal{G}^c)}$ and
    \begin{equation}
        \mathbb{P} (\left |  \Phi_p(\Delta, X) - \mathbb{E}\Phi_p(\Delta, X)   \right | > \delta ) \le 2 e^{-C_1 \delta ^2 p} + 2 e^{- p} \le  e^{-C p} \text{,} \nonumber 
    \end{equation}
    which concludes the proof of lemma \ref{concentration_ineq_for_second_moment_method}.
\end{proof}



\noindent 
Concerning large-p behaviour of $\phi_p(\Delta)$, we have the following characterization. 
\begin{lemma}
    Fix $\alpha \in (0,\infty)$ and $\Delta \in (0, \infty)$. One has
    \begin{equation}
        \lim_{p \to \infty} \phi_p(\Delta, X) = -\frac{\alpha}{2} \ln (1 + \frac{1}{\Delta \alpha}). \nonumber 
    \end{equation}
\end{lemma}

\begin{proof}
We compute the LHS explicitly, namely
    \begin{equation}
        \begin{aligned}
            \phi_p(\Delta, X) &= \frac{1}{p} \ln \mathbb{E}_X \mathbb{E}_{\beta} e^{-\frac{1}{2\Delta} \beta^{\top} X^{\top} X \beta} \\
            &= \frac{1}{p} \ln  \mathbb{E}_{\beta} \mathbb{E}_X e^{-\frac{1}{2\Delta} \operatorname{tr}( X^{\top} X \beta \beta^{\top})} \\
            &= \frac{1}{p} \ln \mathbb{E}_{\beta} \mathbb{E}_X  e^{-\frac{1}{2\Delta} \lVert \beta \rVert^2 (X^{\top} X)_{11}}\\
            &= \frac{1}{p} \ln \mathbb{E}_X e^{-\frac{p}{2\Delta} (X^{\top} X)_{11}} \\
            &= \frac{1}{p} \ln \mathbb{E}_{V\sim \chi^2(n)} e^{-\frac{p}{2\Delta n} V}\text{.}
        \end{aligned} \nonumber 
    \end{equation}
    Note that moment generating function of $\chi^2$ exists for any $t < \frac{1}{2}$, for any $\Delta > 0$ we have
    \begin{equation}
        \phi_p(\Delta, X) = \frac{1}{p} \ln ((1 + \frac{p}{\Delta n})^{-n/2}) \to - \frac{\alpha}{2} \ln(1 + \frac{1}{\Delta \alpha}) \text{.} \nonumber 
    \end{equation}
\end{proof}
A point worth noting is that there is no non-trivial restriction on the value of $\Delta$. Now we proceed to use second moment method to show $\Phi(\Delta, X)$ also converges to the same quantity. Let $\sigma$ and $\beta$ be independent samples from the prior distribution $\pi_0$. Then
\begin{equation}
    \begin{aligned}
    \mathbb{E}(Z_{p}(\Delta, X)^{2})
    &=\mathbb{E}_{X} \mathbb{E}_{\sigma}e^{-\frac{1}{2\Delta}H(\sigma)}\mathbb{E}_{\beta}e^{-\frac{1}{2\Delta}H(\beta)}\\
    &=\mathbb{E}_{X} \mathbb{E}_{\beta,\sigma}e^{-\frac{1}{2\Delta}(H(\sigma)+H(\beta))}\\
    &= \mathbb{E}_{X} \mathbb{E}_{\beta,\sigma} \exp \left(-\frac{1}{2\Delta} \operatorname{tr}\left\{ X^{\top} X\left(\sigma \sigma^{\top} +\beta \beta^{\top} \right)\right\}\right) \\
    &= \mathbb{E}_{X} \mathbb{E}_{\beta,\sigma} \exp\left(-\frac{1}{2\Delta} \left\{ (p + \beta^{\top} \sigma)(X^{\top}X)_{11} + (p - \beta^{\top} \sigma)(X^{\top}X)_{22}    \right\}\right) \text{,}
    \end{aligned} \nonumber 
\end{equation}
where the last line follows from $\operatorname{rank}(\sigma \sigma^{\top} +\beta \beta^{\top}) = 2$ and in fact its non-zero eigenvalues are $\lambda_1 = p + \beta^{\top} \sigma$ and $\lambda_2 = p - \beta^{\top} \sigma$. For $i = 1,2$, let $V_i = (X^{\top} X)_{ii}$. Note that $V_1$ and $V_2$ are iid with $V_i \sim \frac{1}{n} \chi^{2}(n)$. 
Therefore,
\begin{equation}
    \begin{aligned}
    \mathbb{E}(Z_{p}(\Delta, X)^{2}) &= \mathbb{E}_{V} \mathbb{E}_{\beta,\sigma} \exp\left(-\frac{1}{2\Delta} \left\{ (p + \beta^{\top} \sigma)V_1 + (p - \beta^{\top} \sigma)V_2    \right\}\right) \\
    &= \mathbb{E}_{\beta,\sigma} \mathbb{E}_{V} \exp \left( -\frac{1}{2\Delta}(\lambda_1 V_1 + \lambda_2 V_2)  \right) \text{.}
\end{aligned} \nonumber 
\end{equation}
Again, using MGF of $\chi^2$ distribution, we further have
\begin{equation}
    \label{eq:annealed_second_moment_calculation}
    \begin{aligned}
    \mathbb{E}(Z_{p}(\Delta, X)^{2}) &= \mathbb{E}_{\beta,\sigma} \mathbb{E}_{V} \exp \left( -\frac{1}{2\Delta}(\lambda_1 V_1 + \lambda_2 V_2)  \right) \\
    &= \mathbb{E}_{\beta,\sigma}\left \{ (1 + \frac{\lambda_1}{\Delta n})^{-n/2} (1 + \frac{\lambda_2}{\Delta n})^{-n/2} \right \}\\
    & = \mathbb{E}_{\beta,\sigma}\left \{ \left [ (1 + \frac{p}{\Delta n})^2  - (\frac{\beta^{\top} \sigma}{ \Delta n})^2 \right ]^{-n/2}\right \} \text{.}
\end{aligned}
\end{equation}
Let 
\begin{equation}
    \gamma_0 = \frac{4 \mathbb{E}(Z_{p}(\Delta, X)^{2})}{(\mathbb{E}(Z_{p}(\Delta, X)))^{2}} \text{.} \nonumber 
\end{equation}

In order to apply the second moment method, we first establish the following characterization of $\gamma_0$.

\begin{lemma}
\label{lim_calculation_of_log_gamma}
Fix $\alpha \in (0,\infty)$. There exists $\Delta_0>0$ such that for all $\Delta>\Delta_0$,
    \begin{equation}
        \lim_{p \to \infty} \frac{1}{p} \ln \gamma_0 = 0 \text{.} \nonumber
    \end{equation} 
\end{lemma}

\begin{proof}
    Due to symmetry, $ \beta^{\top} \sigma  \stackrel{d}{=} \Sigma_{i = 1}^{p}\beta_i$, equivalently, we can assume $\sigma = (1,1,\dots,1)$ without loss of generality. By Lemma \ref{levys_lemma}, letting $Q =  \beta^{\top} \sigma / p \in (-1,1)$, for $t \in (0,1)$,
\begin{equation}
    \mathbb{P}_{\beta, \sigma}(|Q| > t ) < e^{\pi - p t^2 /4} \text{.} \nonumber 
\end{equation}
Armed with this deviation bound, we derive an upper bound on $\mathbb{E}_{X}(Z_{p}(\Delta, X)^{2})$.  
Setting $r = p / (\Delta n)$, one has
\begin{equation}
    \begin{aligned}
        \mathbb{E}(Z_{p}(\Delta, X)^{2}) &= \mathbb{E}_{\beta,\sigma}\left \{ \left [ (1 + \frac{p}{\Delta n})^2  - (\frac{ \beta^{\top} \sigma }{ \Delta n})^2 \right ]^{-n/2}\right \} \\
        & = \mathbb{E}_Q \left \{ \left [ (1 + r)^2 - (Q r)^2\right ]^{-n/2} \right \}  \coloneqq \mathbb{E}_Q \left [ e^{p f(Q)}\right ] \text{,}
    \end{aligned}
\end{equation}
where $f(t) \coloneqq -( \alpha / 2) \ln \left [ (1+r)^2 - (tr)^2 \right ]$. Let $g(t) = f(t) - t^2 / 4$. 
Note that ${g}^\prime(0) = 0$
and
\begin{equation}
    \begin{aligned}
        {g}^{\prime \prime}(t) = \frac{\alpha r^2 \left [ (1 + r)^2 + t^2 r^2\right ]}{\left [ (1 + r)^2 - t^2 r^2 \right ]^2} - \frac{1}{2} < {g}^{\prime \prime}(1)
        =: - C_Q \text{,}
    \end{aligned}
\end{equation}
for any $t \in (0,1)$. Note that $r \to 0$ for any fixed $\alpha$ as $\Delta \to \infty$, and thus $C_Q >0$ for $\Delta>0$ sufficiently large. 
We have $g(t) \le - C_Q t^2 / 2 + g(0) = - C_Q t^2 /2 - \alpha \ln (1 + r) $. 
By \eqref{eq:annealed_second_moment_calculation}, integration by parts gives 
\begin{equation}
    \begin{aligned}
        \mathbb{E}(Z_{p}(\Delta, X)^{2}) 
        & = - \int_{0}^{1} e^{p f(t)} \mathrm{d} \mathbb{P}_{\beta, \sigma}(| Q | >t)  \\
        & \le - \left [ e^{pf(t)}\mathbb{P}_{\beta,\sigma}( | Q | >t) \Big|_{t=0}^{t = 1}- \int_{0}^{1} \left [ e^{p f(t)} \right ]^\prime e^{\pi -t^2/4}\right ]\mathrm{d} t \\
        & \le - e^{pf(0)} +e^\pi  \int_0^1  \frac{p \alpha r^2 t}{(1 + r)^2 - t^2 r^2} e^{ p g(t)} \mathrm{d} t \\
        & \le - e^{pf(0)} +e^\pi \frac{p\alpha r^2}{(1+r)^2 -r^2} \int_0^1 e^{ p g(t)} \mathrm{d} t \\
        & \le e^{\pi - p \alpha \ln (1+r)} \left [ 1 + C p \int_{0}^{1} e^{- p C_Q t^2 / 2} \mathrm{d} t \right ] \\
        & \le  e^{\pi -p \alpha \ln (1+r) } ( C \sqrt{n} + 1) \text{.}
    \end{aligned} \nonumber 
\end{equation}
Thus
\begin{equation}
    \begin{aligned}
        \frac{1}{p}\ln \gamma_0 
        & \le \frac{1}{p} \ln \left \{ \frac{4e^\pi (1+r)^{-n} (C \sqrt{n}+ 1)}{(1+r)^{-n}}  \right \} \to 0 \text{,}
    \end{aligned}
\end{equation}
as $p \to \infty$.
\end{proof}
For the readers' convenience, we re-state \cite[Lemma 4.1.1]{montanari2013statistical} here.
\begin{lemma}
\label{second_moment_method_lemma}
    \begin{equation}
         \mathbb{P}\left(\left|\Phi_{p}(\Delta, X)- \phi_{p}(\Delta)\right|<\frac{1}{p} \ln \gamma_0 \right) \ge \frac{1}{\gamma_0} \text{.} \nonumber 
    \end{equation}
\end{lemma}
Now we are finally ready to prove Lemma \ref{limit_of_free_energy_without_external_field}.
\begin{proof}[Proof of Lemma \ref{limit_of_free_energy_without_external_field}]
Putting Lemma \ref{concentration_ineq_for_second_moment_method}, Lemma \ref{second_moment_method_lemma}, and Lemma \ref{lim_calculation_of_log_gamma} together, we have for any $\delta \in (0,1)$, there exists a large $p_0 \in \mathbb{N}$ such that for any $p > p_0$
\begin{equation}
    \mathbb{P}\left( \left |  \Phi_p(\Delta, X) - \mathbb{E}\Phi_p(\Delta, X)   \right | \le \delta  \right) + \mathbb{P}\left(\left|\Phi_{p}(\Delta, X)- \phi_{p}(\Delta)\right|< q_p \right) \ge 1 - e^{Cp} + e^{q_p p} > 1 \text{,} \nonumber 
\end{equation}
where $q_p \coloneqq \frac{1}{p}\ln \gamma_0 \to 0$ as $p \to \infty$. It implies
\begin{equation}
    \left |\mathbb{E}\Phi_p(\Delta, X)  - \phi_{p}(\Delta) \right | \le \delta \text{,} \nonumber 
\end{equation}
which further implies
\begin{equation}
    \limsup_{p \to \infty} \left | \mathbb{E}\Phi_p(\Delta, X)  - \phi_{p}(\Delta) \right | \le \delta \text{.} \nonumber 
\end{equation}
Since it holds for any $\delta \in (0,1)$, proof of Lemma \ref{limit_of_free_energy_without_external_field} is therefore complete.
\end{proof}

\noindent 
Now we establish Lemma \ref{limit_of_free_energy_without_external_field_on_a_linear_subspace} based on Lemma \ref{limit_of_free_energy_without_external_field_uniform}.

\begin{proof}[Proof of Lemma \ref{limit_of_free_energy_without_external_field_on_a_linear_subspace}]
    For any $u,v \in \mathbb{R}^p$, let $w_1,w_2,\dots,w_p$ be an orthonormal basis such that $\operatorname{Span}( u, v ) = \operatorname{Span}( w_{p-1}, w_{p})$. Let $\{ \lambda_1 \ge \lambda_2 \ge \dots \ge \lambda_{p-2} \}$ be the eigenvalues of $Q$, where Q is the top left $(p-2)$ by $ (p-2)$ minor of $X^{\top} X$ under this basis. 
    Namely, let $W = (w_1, w_2, \dots, w_p) \in \mathbb{R}^{p \times p}$, $X = W \tilde{X}$ and $\beta = W \tilde{\beta}$, where $\tilde{\beta} = (\tilde{\beta}_1, \tilde{\beta}_2, \dots, \tilde{\beta}_p)$, with $\tilde{\beta}_p = \tilde{\beta}_{p-1} = 0$ for $\beta \in \operatorname{Span}( u, v )^{\perp}$. Equivalently, $\tilde{X}$ and $\tilde{\beta}$ are nothing but $X$ and $\beta$ under this new basis and $Q = (\tilde{X}^{\top} \tilde{X})_{1:(p-2),1:(p-2)}$.

    \begin{equation}
    \label{proof_of_limit_of_free_energy_without_external_field_on_a_linear_subspace_1}
        \begin{aligned}
        \mathbb{E}_{\beta}^{\operatorname{Span}( u, v )^{\perp}} [ \exp (-H(\beta))] 
        &= \mathbb{E}_{\beta}^{\operatorname{Span}( u, v )^{\perp}}[ \exp \left (-\frac{1}{2\Delta} \beta^{\top} X^{\top} X \beta \right )] \\
        &=  \mathbb{E}_{\beta}^{\operatorname{Span}( u, v )^{\perp}} \left [ \exp \left ( -\frac{1}{2 \Delta} \tilde{\beta}^{\top} \tilde{X}^{\top} \tilde{X} \tilde{\beta} \right ) \right ] \\
        &= \mathbb{E}_{\beta}^{\operatorname{Span}( u, v )^{\perp}} \left [ \exp \left ( -\frac{1}{2 \Delta} \tilde{\beta}_{1:(p-2)}^{\top} Q \tilde{\beta}_{1:(p-2)} \right ) \right ] \\
        &= \mathbb{E}^{p-2}_{\tilde{\beta}_{1:(p-2)}} \left [ \exp \left ( -\frac{1}{2 \Delta}   \sum_{i=1}^{p-2} \lambda_i \tilde{\beta}_i^2 \right ) \right ] \\
        &= \mathbb{E}^{p-2}_{\beta} \left [ \exp \left ( -\frac{1}{2 \Delta}   \sum_{i=1}^{p-2} \lambda_i \beta_i^2 \right ) \right ]
        \text{,}
        \end{aligned}
    \end{equation}
    where $\mathbb{E}^{p-2} $ denotes expectation under $\operatorname{Unif}(S^{(p-3)}(\sqrt{p}))$. Note that the second last line is due to the distributional invariance of $\tilde{\beta}$ under rotations, which itself is inherited from that of $\beta \sim \operatorname{Unif}(S^{p-1}(\sqrt{p}))$ and the last line above is nothing but a change of notation. 
    Similarly, let $S$ be the top left $(p-2)$ by $(p-2)$ minor of $X^{\top} X $ when written in the standard basis $\{e_1,e_2, \dots, e_p\}$. let $s_1 \ge s_2 \ge \dots \ge s_{p - 2}$ be the corresponding eigenvalues of $S$. For $\beta \in \mathbb{R}^{p-2}$, by Lemma \ref{limit_of_free_energy_without_external_field} we have
    \begin{equation}
    \label{proof_of_limit_of_free_energy_without_external_field_on_a_linear_subspace_2}
        \mathbb{E}^{p-2}_{\beta} e^{- \frac{1}{2\Delta} \beta^{\top} S \beta} = e^{p(-\frac{\alpha}{2} \ln (1+\frac{1}{\Delta}) + o(1))} \text{,}
    \end{equation}
    in which the $o(1)$ term is uniform over any $\Delta \ge \Delta_0$. At the same time
    \begin{equation}
    \label{proof_of_limit_of_free_energy_without_external_field_on_a_linear_subspace_3}
        \mathbb{E}^{p-2}_{\beta} e^{- \frac{1}{2\Delta} \beta^{\top} S \beta} = \mathbb{E}^{p-2}_{\beta} e^{-\frac{1}{2 \Delta} \sum_{i=1}^{p-2} s_i \beta_i^2}\text{.}
    \end{equation}
    By the eigenvalue interlacing inequality \cite{horn2012matrix}, we have almost surely
    \begin{equation}
        \sup_{1 = 1,2,..,p} |\lambda_i - s_i| \to 0 \text{.} \nonumber 
    \end{equation}
    Therefore almost surely
    \begin{equation}
    \label{proof_of_limit_of_free_energy_without_external_field_on_a_linear_subspace_4}
        \sup_{\beta \in S^{p-2}(\sqrt{p})} \frac{1}{p} |\sum_{i=1}^{p-2} \lambda_i \beta_i^2 - \sum_{i=1}^{p-2} s_i \beta_i^2| \le \frac{ \lVert \beta \rVert^2}{p} \sup_{1 = 1,2,..,p} |\lambda_i - s_i| \to 0 \text{.}
    \end{equation}
    Combining (\ref{proof_of_limit_of_free_energy_without_external_field_on_a_linear_subspace_1},\ref{proof_of_limit_of_free_energy_without_external_field_on_a_linear_subspace_2},\ref{proof_of_limit_of_free_energy_without_external_field_on_a_linear_subspace_3},\ref{proof_of_limit_of_free_energy_without_external_field_on_a_linear_subspace_4}), we obtain
    \begin{equation}
    \label{eigen_values_uniform_convergence}
        \mathbb{E}^{\operatorname{Span}( u, v )^{\perp}} e^{-\frac{1}{2 \Delta}\beta^{\top} X^{\top} X \beta} =  e^{p(-\frac{\alpha}{2} \ln (1+\frac{1}{\Delta}) + o(1))} \text{.}
    \end{equation}
    Note the $o(1)$ term is uniform with respect to not only $\Delta \ge \Delta_0$ but also the choice of $u$ and $v$, since the convergence in (\ref{eigen_values_uniform_convergence}) is independent of $u$ and $v$, which concludes the proof.
\end{proof}


\section{A matching upper bound on the free energy}
\label{sec:upper_bound}

We derive an upper bound to the TAP free energy in terms of the log-partition function in this section. We collect this bound in Theorem \ref{upper_bound} below. 
To establish this result, we first introduce a model perturbation and discuss overlap concentration under the perturbed model in Section \ref{sec:perturbation}. These results are subsequently used to establish Theorem \ref{upper_bound} in Section \ref{sec:proof_upperbound}.  

\begin{theorem}[Upper bound]
\label{upper_bound}
Fix $\Delta>0$ and $\alpha \in (0,\infty)$. 
    For any $\eta>0$ as $p \to \infty$,  
    \begin{equation}
        \mathbb{P} \left (  \frac{1}{p} \ln \mathcal{Z}_p \ge \sup_{a \in \mathbb{R}^p, \lVert a \rVert \le \sqrt{p}} f_{\text{TAP}}(a) + \eta  \right ) \to 0 \text{.} \nonumber 
    \end{equation}
\end{theorem}

\subsection{Perturbed model and overlap concentration}
\label{sec:perturbation}

Assume that apart from $X$ and $y$, we observe the following additional side information
    \begin{equation}
        y^{\text{Pert}} = \sqrt{\lambda_0 \epsilon_p} \beta_0 + Z \text{,} \nonumber 
    \end{equation}
with $Z \sim \mathcal{N}(0,I_p)$, $\lambda_0 \in [1/2, 1]$, $\epsilon_p \to 0$, and $p \epsilon_p \to +\infty$. We refer to $(y^{\text{Pert}}, y, X)$ as the perturbed model.
Having observed $(y^{\text{Pert}}, y, X)$, one constructs the posterior distribution 
\begin{equation}
\label{eq:perturbed_posterior} 
    \frac{d \mathbb{P}}{d \pi_0}(\beta | y^{\text{Pert}}, y, X) = \frac{1}{\mathcal{Z}_p^{\text{Pert}}} e^{- \mathcal{H}_p^{\text{Pert}}(\beta)}, 
\end{equation}
where $\mathcal{H}_p^{\text{Pert}} := \mathcal{H}_p(\beta) + \mathcal{H}_p^{\text{Gauss}}(\beta)$ is the  \textit{Hamiltonian} of the perturbed model, 
and $\mathcal{H}_p^{\text{Gauss}}(\beta) \coloneqq - \lambda_0 \epsilon_p \beta_0^{\top} \beta - \sqrt{\lambda_0 \epsilon_p} Z^{\top} \beta +  \lambda_0 \epsilon_p \lVert \beta \rVert^2 / 2$ is the  Hamiltonian of the Gaussian side channel. 
The normalizing constant of this posterior distribution is 
\begin{equation}
\label{eq:normalizing_constant} 
    \mathcal{Z}^{\text{Pert}}_p \coloneqq \int_{S^{p-1}(\sqrt{p})} e^{-\mathcal{H}^{\text{Pert}}(\beta)} d\pi_0(\beta) \text{.}
\end{equation}
Finally, the free energy of the perturbed model is
\begin{equation}
\label{eq:perturbed_free_energy} 
    F^{\text{Pert}}_p(\lambda_0) \coloneqq  \frac{1}{p} \ln \mathcal{Z}^{\text{Pert}}_p(y^{\text{Pert}},y,X) \text{.}
\end{equation}
In the subsequent discussion, we abuse notation slightly, and use $\langle \cdot \rangle$ to also denote the expectation under the perturbed posterior \eqref{eq:perturbed_posterior}. 
The notion of overlap will play a crucial role in our subsequent discussion. 
\begin{definition}
\label{defn:overlap} 
Let $\beta^1$, $\beta^2$ be two iid samples from the posterior distribution \eqref{eq:perturbed_posterior}. We refer to $R_{12}:=  (\beta^1)^{\top} \beta^2 / p $ as the \emph{overlap} between two replicas. 
\end{definition}

\begin{theorem}
\label{overlap_concentration}
    For any $\Delta > 0$ and $\alpha \in (0,\infty)$, as $p \to \infty$,
    \begin{equation}
        \frac{1}{p} \left | \mathbb{E} \ln \mathcal{Z}_p^{\textrm{Pert}} -   \mathbb{E} \ln \mathcal{Z}_p     \right | \le O(\epsilon_p) \text{.} \nonumber 
    \end{equation}
    Furthermore, assuming $\lambda_0 \sim \operatorname{Unif}(\frac{1}{2},1)$, there exist $C_1, C_2>0$ such that
    \begin{equation}
    \label{overlap_concentration_main_equation}
        \mathbb{E}_{\lambda_0} \mathbb{E} \langle (R_{1,2} - \mathbb{E}\langle R_{1,2} \rangle )^2 \rangle \le \frac{C_1}{\epsilon_p} \Big(\frac{v_p}{p\epsilon_p} + \frac{1}{p} \Big)^{1/3} \text{,}
    \end{equation}
    where $v_p \coloneqq p \sup_{\lambda_0 \in [\frac{1}{2},1]} \{ \mathbb{E} (F_p^{\text{Pert}}(\lambda_0) -  \mathbb{E}F_p^{\text{Pert}}(\lambda_0))^2  \}$ and $v_p < C_2$ for any $p\geq 1$. 
\end{theorem}
In our application of Theorem \ref{overlap_concentration} we will choose $\epsilon_p \to 0$ slow enough such that the RHS in (\ref{overlap_concentration_main_equation}) converges to zero as $p \to \infty$. Our next result characterizes the limiting behavior of $\mathbb{E}[\langle R_{1,2} \rangle ]$. 
\begin{lemma}
\label{lower_and_upper_bounds_on_overlap}
    For any $\Delta, \alpha > 0$, as $p \to \infty$, 
    there exists $C>0$ depending on $\alpha$ and $\Delta$ such that 
    \begin{equation}
        1 -C \ge \limsup_{p \to \infty}  \mathbb{E}[\langle R_{1,2} \rangle] \ge  \liminf_{p \to \infty} \mathbb{E}[\langle R_{1,2} \rangle] \ge 0 \text{.} \nonumber 
    \end{equation}
\end{lemma}
Note that $\mathbb{E}[\langle R_{1,2} \rangle] \in [-1,1]$---thus any subsequence has a further subsequence such that $\mathbb{E}[\langle R_{1,2} \rangle]$ converges to a limit along the subsequence. Further, Lemma \ref{lower_and_upper_bounds_on_overlap} implies that this subsequential limit lies in $[0,1-C]$. 
Armed with these results, we relate the log-partition function to that of a\emph{replicated} system, where the overlaps have been restricted to have a specific value. 
\begin{lemma}
\label{upperbound_intermediate_lemma_1}
        For $m \geq 1$ and $c>0$, define 
        \begin{equation}
            \label{definition_of_S}
            \mathcal{S}_p(m,c) = \left \{(\beta^1, \beta^2,\dots,\beta^{m}): \left | \frac{1}{p}  (\beta^i)^{\top} \beta^j  - \mathbb{E} \langle R_{1,2} \rangle \right | \le c, \forall i \ne j \right \} \text{.}
        \end{equation}
        There exists $c_p \to 0$ and $m_p \to \infty $, such that 
            \begin{equation}
                \frac{1}{p m_p} \ln \left ( \int e^{- \sum_{i} \mathcal{H}(\beta^i)} \mathbf{1}_{\{ (\beta^1, \beta^2,\dots,\beta^{m_p}) \in \mathcal{S} \}} \prod_{i = 1}^{m_p} \mathrm{d} \pi_0(\beta^i) \right )  = \frac{1}{p} \ln \mathcal{Z}_p + o(1). \nonumber 
            \end{equation}
        In the equation above and henceforth, we suppress the dependence of $\mathcal{S}$ on $p$, $c_p$ and $m_p$ for the sake of notational simplicity. 
\end{lemma}
\begin{proof}
\label{proof_of_upperbound_intermediate_lemma_1}
    Using Theorem \ref{overlap_concentration},  we can choose a sequence $\{\lambda_0^p: p \ge 1\}$ such that for any $p\ge1$ we have
    \begin{equation}
        \mathbb{E} \langle (R_{1,2} - \mathbb{E}\langle R_{1,2} \rangle )^2 \rangle \le a_p \text{,} \nonumber 
    \end{equation}
    where $a_p \to 0$ as $p \to \infty$. For any positive sequence $\{ b_p: p\ge 1\}$, Markov inequality yields
    \begin{equation}
        \mathbb{P}\left (  \langle (R_{1,2} - \mathbb{E}\langle R_{1,2} \rangle )^2 \rangle > b_p      \right ) \le \frac{a_p}{b_p} \text{.} \nonumber 
    \end{equation}
    We choose $\{ b_p: p\ge 1\}$ such that $b_p \to 0$, $a_p / b_p \to 0$. Let $m_p \ge 1$ (to be chosen later) and let $\beta^1, \beta^2,\dots,\beta^{m_p}$ be iid samples from the posterior distribution \eqref{eq:perturbed_posterior}.  Define $R_{i,j} =  (\beta^i)^{\top} \beta^j / p $ for $i,j \in \{ 1, 2, \dots, m_p\}$ to be the overlaps among the samples. By symmetry, we have $\langle R_{1,2} \rangle = \langle R_{i,j} \rangle$ and $\mathbb{E}\langle R_{1,2} \rangle = \mathbb{E}\langle R_{i,j} \rangle$ for any $i \ne j$. Define the event $\mathcal{C}_p \coloneqq \{ (y^{\text{Pert}}, y , X): \langle (R_{1,2} - \mathbb{E}\langle R_{1,2} \rangle )^2 \rangle  > b_p\}$. Note that $\mathbb{P}(\mathcal{C}_p) \le a_p / b_p \to 0$. On $\mathcal{C}_p^c$ we have
    \begin{equation}
        \langle (R_{1,2} - \mathbb{E}\langle R_{1,2} \rangle )^2 \rangle \le b_p \text{.} \nonumber 
    \end{equation}
    Using Chebychev inequality, we have, on $\mathcal{C}_p^c$
    \begin{equation}
        \langle \mathbf{1}_{\{ \left | R_{1,2} - \mathbb{E}\langle R_{1,2} \rangle     \right | > c_p \}} \rangle \le \frac{b_p}{c_p^2} \text{.} \nonumber 
    \end{equation}
By union bound, on $\mathcal{C}_p^c$
    \begin{equation}
        \langle \mathbf{1}_{\mathcal{S}^c} \rangle \le \sum_{i \ne j} \langle \mathbf{1}_{\{ \left | R_{i,j} - \mathbb{E}\langle R_{i,j} \rangle \right | > c_p \}} \rangle = \sum_{i \ne j} \langle \mathbf{1}_{\{ \left | R_{1,2} - \mathbb{E}\langle R_{1,2} \rangle \right | > c_p \}} \rangle \le \frac{m_p^2 b_p}{c_p^2} \text{.} \nonumber 
    \end{equation}
    We choose $m_p \to \infty$ and $c_p \to 0$ such that the RHS above will converge to $0$ as $p \to \infty$. This implies, on $\mathcal{C}_p^c$,
    \begin{equation}
    \label{eq:proof_of_upperbound_intermediate_lemma_1_1}
        1 \ge \langle \mathbf{1}_S \rangle = \frac{1}{(\mathcal{Z}_p^{\text{Pert}})^{m_p}} \int e^{- \sum_{i} \mathcal{H}^{\text{Pert}}(\beta^i)} \mathbf{1}_{\{ (\beta^1, \beta^2,\dots,\beta^{m_p}) \in \mathcal{S} \}} \prod_{i = 1}^{m_p} d \pi_0(\beta^i) \ge 1 - \frac{m_p^2 b_p}{c_p^2} = 1 - o(1)\text{.}
    \end{equation}
    Since $\lim_{p\to \infty} \mathbb{P}_{Z}( \lVert Z \rVert < 2\sqrt{p}) = 1 $, we have with high probability with respect to $Z$
    \begin{equation}
        \begin{aligned}
            \left | \mathcal{H}(\beta) - \mathcal{H}^{\text{Pert}}(\beta) \right | &= \left | \mathcal{H}^{\text{Gauss}}(\beta) \right |\\
            &\le \left | \lambda_0 \epsilon_p \beta_0^{\top} \beta \right | + \left |\sqrt{\lambda_0 \epsilon_p} Z^{\top} \beta \right | + \left | \frac{1}{2} \lambda_0 \epsilon_p \lVert \beta \rVert^2 \right |\\
            &\le O( \sqrt{\epsilon_p} ) \lambda_0 p \text{.}    \quad \quad \quad \quad \quad 
        \end{aligned} \nonumber 
    \end{equation}
    This implies that with high probability with respect to $Z$,
    \begin{equation}
    \label{eq:proof_of_upperbound_intermediate_lemma_1_2}
        \left | \ln \mathcal{Z}_p^{\text{Pert}} - \ln \mathcal{Z}_p\right | \le O( \sqrt{\epsilon_p} ) \lambda_0 p \text{,}
    \end{equation}   
    and 
    \begin{equation}
    \label{eq:proof_of_upperbound_intermediate_lemma_1_3}
        \left | \frac{1}{p m_p} \ln \left ( \frac{\int e^{- \sum_{i} \mathcal{H}^{\text{Pert}}(\beta^i)} \mathbf{1}_{\{ (\beta^1,\dots,\beta^{m_p}) \in \mathcal{S} \}} \prod_{i = 1}^{m_p} \mathrm{d} \pi_0(\beta^i) }{  \int e^{- \sum_{i} \mathcal{H}(\beta^i)} \mathbf{1}_{\{ (\beta^1,\dots,\beta^{m_p}) \in \mathcal{S} \}} \prod_{i = 1}^{m_p} \mathrm{d} \pi_0(\beta^i) }  \right ) \right | \le O( \sqrt{\epsilon_p} ) \text{.}
    \end{equation}
    Putting  (\ref{eq:proof_of_upperbound_intermediate_lemma_1_1}, \ref{eq:proof_of_upperbound_intermediate_lemma_1_2}, \ref{eq:proof_of_upperbound_intermediate_lemma_1_3}) together, with high probability  we have
    \begin{equation}
        \frac{1}{p m_p} \ln \left ( \int e^{- \sum_{i} \mathcal{H}(\beta^i)} \mathbf{1}_{\{ (\beta^1, \beta^2,\dots,\beta^{m_p}) \in \mathcal{S} \}} \prod_{i = 1}^{m_p} \mathrm{d} \pi_0(\beta^i) \right )  = \frac{1}{p} \ln \mathcal{Z}_p + O( \sqrt{\epsilon_p} ) \text{.} \nonumber 
    \end{equation}
    Recalling $\epsilon_p \to 0$ as $p \to \infty$, the claim follows. 
\end{proof}

\subsection{Proof of Theorem \ref{upper_bound}}
\label{sec:proof_upperbound} 
We prove Theorem \ref{upper_bound} in this section. To this end, recall that $\mathcal{H}(\beta) = -\frac{1}{2\Delta} \lVert y - X\beta \rVert^2$. Let $\beta^1,\cdots, \beta^{m_p}$ be iid samples from the posterior distribution, and set $a =  \sum_{i = 1}^{m_p} \beta^i / m_p$ to be the sample average. This implies
\begin{equation}
    \begin{aligned}
    \frac{1}{m_p} \sum_{i = 1}^{m_p} \lVert y - X\beta^i \rVert^2 &= \lVert y - X a\rVert^2 + \frac{1}{m_p}\sum_{i = 1}^{m_p} \lVert X (\beta^i - a)\rVert^2 -  \frac{2}{m_p}   (X^{\top}(y - Xa))^{\top}\sum_{i = 1}^{m_p} (\beta^i - a)  \\
    & = \lVert y - X a\rVert^2 + \frac{1}{m_p}\sum_{i = 1}^{m_p} \lVert X (\beta^i - a)\rVert^2
    \text{,}
    \end{aligned} \nonumber 
\end{equation}
since $\sum_{i = 1}^{m_p} (\beta^i - a) = 0$. Thus with high probability,

    \begin{align}
        \frac{1}{p} \ln \mathcal{Z}_p &=  \frac{1}{ p m_p} \ln \left ( \int e^{-\frac{m_p}{2\Delta} \lVert y - Xa\rVert^2  -\frac{1}{2\Delta} \sum_{i=1}^{m_p} \lVert X(\beta^i - a)\rVert^2} \mathbf{1}_{\{ (\beta^1, \beta^2,\dots,\beta^{m_p}) \in \mathcal{S} \}} \prod_{i = 1}^{m_p} \mathrm{d} \pi_0(\beta^i) \right ) + o(1) \nonumber \\
        & \le \frac{1}{ p m_p} \ln \left ( \int e^{-\frac{m_p}{2\Delta} \lVert y - Xa\rVert^2  -\frac{1}{2\Delta} \sum_{i=2}^{m_p} \lVert X(\beta^i - a)\rVert^2} \mathbf{1}_{\{ (\beta^1, \beta^2,\dots,\beta^{m_p}) \in \mathcal{S} \}} \prod_{i = 1}^{m_p} \mathrm{d} \pi_0(\beta^i) \right )+ o(1) \nonumber \\
        &:= R_p + o(1). \label{eq:r_p} 
    \end{align}
Our next result, Lemma \ref{upperbound_intermediate_lemma_2} will be critical for the proof of Theorem \ref{upper_bound}.  
We defer its proof to the next subsection. 
\begin{lemma}
\label{upperbound_intermediate_lemma_2}
    Fix $\Delta >0$, $\alpha \in (0,\infty)$ and any $\eta > 0$. For any subsequence $\{p_k: k \geq 1\}$ such that $\mathbb{E}[\langle R_{1,2} \rangle]$ converges to a positive constant, there exist  $c' >0$  and  $c>0$, $c+c'<1$,  such that  as $k \to \infty$,
    \begin{equation}
        \mathbb{P} \left (   R_{p_k}  \le \sup_{a \in \mathbb{R}^{p_k} , c'\sqrt{p_k} \le \lVert a \rVert \leq  (1 - c)\sqrt{p_k}} f_{\text{TAP}}(a) + \eta \right ) \to 1 \text{.} \nonumber 
    \end{equation}
\end{lemma}
\noindent 
Before proceeding further, we prove Theorem \ref{upper_bound}
assuming Lemma \ref{upperbound_intermediate_lemma_2}. 

\begin{proof}[Proof of Theorem \ref{upper_bound}]
Assume, if possible, that there exists a subsequence $p_k$ and $\epsilon, \eta>0$ such that for all $k \geq 1$, 
\begin{align}
     \mathbb{P} \left (  \frac{1}{p_k} \ln \mathcal{Z}_{p_k} \ge \sup_{a \in \mathbb{R}^{p_k}, \lVert a \rVert \le \sqrt{p_k}} f_{\text{TAP}}(a) + \eta  \right ) \geq \epsilon. \nonumber 
\end{align}
\noindent 
We can extract a further subsequence such that $\mathbb{E}[\langle R_{1,2} \rangle]$ converges along the subsequence. For notational simplicity, we refer to this subsequence as $\{p_k : k \geq 1\}$ as well. 

Consider first the case $\lim \mathbb{E}[\langle R_{1,2} \rangle]>0$. Combining \eqref{eq:r_p} and Lemma \ref{upperbound_intermediate_lemma_2}, we obtain that along this subsequence, 
\begin{align}
    \mathbb{P}\Big( \frac{1}{p_k} \ln \mathcal{Z}_{p_k} \ge \sup_{a \in \mathbb{R}^{p_k}, \lVert a \rVert \le \sqrt{p_k}} f_{\text{TAP}}(a) + \eta  \Big) \to 0 . \nonumber 
\end{align}
This is a contradiction, and completes the proof in this case.

We now turn to the case where $\lim \mathbb{E}[\langle R_{1,2}\rangle] =0$. In this case, 
\begin{align}
   \sup_{a \in \mathbb{R}^{p_k}, \|a\|\leq \sqrt{p_k}}  f_{\text{TAP}}(a) &\geq f_{\text{TAP}}(0)= - \frac{1}{2\Delta p_k} \| y \|^2 - \frac{\alpha}{2} \ln \Big( 1 + \frac{1}{\Delta \alpha}  \Big). \label{eq:zero_bound}  
\end{align}
On the other hand, setting $a =  \sum_i \beta^i / m_p$, we note that for $(\beta^1, \cdots, \beta^{m_p}) \in \mathcal{S}$, 
   $  \| a \|^2 / p \leq  1 / ( m_p^2 ) + c_p$, 
where we use that $\lim \mathbb{E}[\langle R_{1,2} \rangle] =0$. Further, using \eqref{eq:r_p}, we have 
\begin{align}
    \frac{1}{p_k} \ln \mathcal{Z}_{p_k} &\leq  \frac{1}{ p_k m_{p_k}} \ln \left ( \int e^{-\frac{m_{p_k}}{2\Delta} \lVert y - Xa\rVert^2  -\frac{1}{2\Delta} \sum_{i=2}^{m_{p_k}} \lVert X(\beta^i - a)\rVert^2} \mathbf{1}_{\{ (\beta^1, \beta^2,\dots,\beta^{m_{p_k}}) \in \mathcal{S} \}} \prod_{i = 1}^{m_{p_k}} \mathrm{d} \pi_0(\beta^i) \right )+ o(1) \nonumber \\
    &\leq \frac{1}{ p_k m_{p_k}} \ln \left ( \int e^{-\frac{m_{p_k}}{2\Delta} \lVert y\rVert^2  -\frac{1}{2\Delta} \sum_{i=2}^{m_{p_k}} \lVert X\beta^i \rVert^2} \mathbf{1}_{\{ (\beta^1, \beta^2,\dots,\beta^{m_{p_k}}) \in \mathcal{S} \}} \prod_{i = 1}^{m_{p_k}} \mathrm{d} \pi_0(\beta^i) \right )+ o(1), \nonumber  \nonumber
\end{align}
where we have used the fact that $\|a\| = o(\sqrt{p})$. In turn, this implies 
\begin{align}
    \frac{1}{p_{k}} \ln \mathcal{Z}_{p_k} &\leq - \frac{\|y\|^2}{2\Delta p} + \frac{m_{p_k}-1}{p_k m_{p_k}} \ln \Big( \int \exp(-\frac{1}{2\Delta} \|X \beta\|^2) \mathrm{d} \pi_0(\beta) \Big) + o(1) \nonumber  \\
    &\leq - \frac{\|y \|^2}{2\Delta p_k} - \frac{\alpha}{2} \ln \Big( 1 + \frac{1}{\Delta \alpha}\Big) +o(1),  \nonumber 
\end{align}
where the last step uses Lemma \ref{limit_of_free_energy_without_external_field}. Combining this bound with \eqref{eq:zero_bound} we immediately obtain that along this subsequence, 
\begin{align}
    \mathbb{P}\Big( \frac{1}{p_k} \ln \mathcal{Z}_{p_k} \ge \sup_{a \in \mathbb{R}^{p_k}, \lVert a \rVert \le \sqrt{p_k}} f_{\text{TAP}}(a) + \eta  \Big) \to 0 \nonumber
\end{align}
which is a contradiction. This completes the proof. 
\end{proof}

It remains to prove Lemma \ref{upperbound_intermediate_lemma_2}. To this end, we will assume without loss of generality that as $p \to \infty$
\begin{equation}
\label{eq:wlog_convergence_of_E_R_12}
    \mathbb{E}[\langle R_{1,2} \rangle] \to C_0 > 0 \text{.}
\end{equation} 
It will be useful to express the uniform distribution on $S^{p-1}(\sqrt{p})$ in terms of the gaussian measure. Specifically, let $g^1, g^2, \dots, g^{m_p} \sim \mathcal{N}(0, I_p)$ be iid Gaussian vectors and set $\beta^{i} = g^i \sqrt{p}/\|g^{i}\|_2$. The rotational symmetry of the gaussian measure ensures that the $\beta^{i}$ vectors are iid uniform on $S^{p-1}(\sqrt{p})$. 
Set $b =  \sum_{i = 1}^{m_p} g^i / m_p$ to  be sample mean of the $g^{i}$ vectors, and let  $\tilde{b} \coloneqq b \sqrt{p} \lVert b \rVert $ be the scaled version of $b$ with radius $\sqrt{p}$.  
We introduce a counterpart of $\mathcal{S}$ for the $g^i$-vectors: 
\begin{align}
\label{definition_of_tilde_S}
    \tilde{\mathcal{S}}_{\epsilon} \coloneqq \left \{ (g^1, g^2, \dots, g^{m_p}): (\sqrt{p} \frac{g^1}{ \lVert g^1 \rVert} , \sqrt{p} \frac{g^2}{ \lVert g^2 \rVert}, \dots, \sqrt{p} \frac{g^{m_p}}{ \lVert g^{m_p} \rVert} ) \in \mathcal{S}, \lVert g^i \rVert \in (\sqrt{p}(1- \epsilon), \sqrt{p} (1 + \epsilon)), \forall i \right \} \text{.}
\end{align}
We first collect some geometric properties of the configurations in $\mathcal{S}$ and $\tilde{\mathcal{S}}_{\epsilon}$. 
The following lemma shows that, $\beta^i - a$ and $a$ are approximately orthogonal on the set $\mathcal{S}$. Its proof follows directly from  the definition of $\mathcal{S}$ and Lemma \ref{lower_and_upper_bounds_on_overlap}.
\begin{lemma}
\label{inner_product_between_beta_centered_and_a}
There exists $C>0$ depending only on $\Delta$ and $\alpha$ such that for any $(\beta^1, \beta^2,\dots,\beta^{m_p}) \in \mathcal{S}$, we have for any $i \in \{1,2,...,m_p \}$,
    \begin{equation}
        \left |  (\beta^i - a)^{\top} \tilde{a}  \right | \le C p c_p \text{,} \nonumber 
\end{equation}
where $a =  \sum_{i = 1}^{m_p} \beta^i / m_p $ and $\tilde{a} \coloneqq a\sqrt{p} / \lVert a \rVert$ is a normalized version of $a$ with radius $\sqrt{p}$.
\end{lemma}
\noindent 
We also establish an analogous property for the configurations in $\tilde{\mathcal{S}}_{\epsilon}$ for small $\epsilon>0$. 
\begin{lemma}
\label{inner_product_between_g_and_b}
    There exist $C_1,C_2,C_3>0$ such that if $\epsilon \in (0, C_1)$ and $g = (g^1, g^2, \dots, g^{m_p}) \in \tilde{\mathcal{S}}_{\epsilon}$, then
    \begin{equation}
        \left |  (g^i - b)^{\top} \tilde{b}  \right | \le ( C_2 c_p + C_3 \epsilon)  p  \text{.} \nonumber 
    \end{equation}
\end{lemma}
\begin{proof}
    Note that, for any $ g \in \mathcal{\tilde{S}}_{\epsilon}$
    \begin{equation}
    \label{norm_of_b_minus_a}
        \begin{aligned}
            \lVert b - a \rVert 
            &= \frac{1}{m_p} \left \lVert \sum_{i = 1}^{m_p} (1 - \sqrt{p}/\lVert g^i \rVert) g^i \right  \rVert \le \frac{1}{m_p} \sum_{i = 1}^{m_p} \left \lVert  (1 - \sqrt{p}/\lVert g^i \rVert) g^i \right  \rVert = O(\epsilon \sqrt{p}) \text{.}
        \end{aligned}
    \end{equation}
   Similarly we also have $\lVert g^i - \beta^i \rVert = O(\epsilon \sqrt{p})$ on $\tilde{\mathcal{S}}_{\epsilon}$.
    By Lemma \ref{inner_product_between_beta_centered_and_a} and triangle inequality,
    \begin{equation}
        \begin{aligned}
            |  (g^i - b)^{\top} b | 
            & \le | (g^i - b)^{\top} a| + | (g^i - b)^{\top}( b - a) |\\
            & \le | (g^i - \beta^i)^{\top} a | + | (\beta^i - a)^{\top} a| + | (a - b)^{\top} a| + \lVert g^i - b \rVert \lVert b - a \rVert \\
            & \le \lVert a \rVert (\lVert g^i - \beta^i \rVert + \lVert a - b \rVert + \lVert \beta^i - a \rVert) + O(\epsilon p) \\
            &\le  O(\epsilon p) + 2 p c_p\\
            &= (2 c_p + O(\epsilon)) p \text{.} \nonumber 
        \end{aligned}
    \end{equation}
    Moreover, note that Lemma \ref{lower_and_upper_bounds_on_overlap} together with \eqref{eq:wlog_convergence_of_E_R_12} imply there is some positive constant $C$ such that if $(\beta^1, \dots, \beta^{m_p}) \in \mathcal{S}$ then $(1 - C)\sqrt{p} \ge \lVert a \rVert \ge C \sqrt{p}$. Thus, for small enough $\epsilon$, 
    there always exists positive constant $C'$ such that $(1 - C') \ge \sqrt{p}\lVert b \rVert \ge C' \sqrt{p}$, which finishes the proof.
\end{proof}
Let $w \sim \mathcal{N}(0, I_p / m_p)$, independent of $\{g^i: 1 \leq i \leq m_p\}$. Let $h^i = g^i + w$ and $q^i = h^i - b$.  Observe that $b, q^2, q^3  \dots, q^{m_p}$ are independent. Moreover, conditioning on $b$, $h_2, h_3, \dots, h_{m_p} $ are conditionally independent.
The following lemma shows that adding a small perturbation $w$ has a negligible effect.
\begin{lemma} There exists $C>0$ such that with high probability (with respect to $X$), as $p \to \infty$
    \begin{equation}
        \mathbb{P}_{g,w} \left ( \left | \sum_{i = 2}^{m_p} (\lVert X (g^i - b) \rVert^2  - \lVert X (g^ i - b + w) \rVert^2) \right |  \le C {m_p}^{1/2} p    \right ) \to 1 \text{.} \nonumber 
    \end{equation}
\end{lemma}
\begin{proof}
We start with
    \begin{equation}
        \begin{aligned}
           & \left | \sum_{i = 2}^{m_p} \left ( \lVert X (g^i - b) \rVert^2   - \lVert X (g^i - b + w) \rVert^2 \right ) \right | \\
            & \le (m_p - 1) \lVert X w \rVert^2 + 2 \sum_{i = 2}^{m_p} \left | w^{\top} X^{\top} X (g^i - b) \right | \\
            & \le (m_p - 1) \lVert X^{\top} X \rVert_2 \lVert w \rVert^2 + 2 \sum_{i = 2}^{m_p} \lVert X^{\top} X \rVert_2 \lVert w \rVert \lVert g^i - b \rVert.
        \end{aligned} \nonumber 
    \end{equation}
    Since $\lVert w \rVert^2 = O(p/m_p)$, $\lVert g^i - b  \rVert = O(\sqrt{p})$, and $\lVert X^{\top} X \rVert_2 = O(1)$ with high probability, we have as $p \to \infty$
    \begin{equation}
        \mathbb{P}_{g,w} \left (
    (m_p - 1) \lVert X^{\top} X \rVert^2 \lVert w \rVert^2 \le C p \right ) \to 1 \text{,} \nonumber 
    \end{equation}
    and
    \begin{equation}
        \mathbb{P}_{g,w} \left ( \left | 2 \sum_{i = 2}^{m_p} \lVert X^{\top} X \rVert_2 \lVert w \rVert \lVert g^i - b \rVert \right | \le C m_p^{1/2} p  \right ) \to 1 \text{.} \nonumber 
    \end{equation}
    The lemma follows upon combining these ingredients.
\end{proof}

\subsection{Proof of Lemma \ref{upperbound_intermediate_lemma_2}}
\begin{proof}[Proof of Lemma \ref{upperbound_intermediate_lemma_2}]
Throughout this proof, we assume without loss of generality that $\mathbb{E}[\langle R_{1,2} \rangle]$ converges to a positive constant. 
Fix any $\epsilon \in (0, C_1)$, where $C_1$ is defined in the statement of Lemma \ref{inner_product_between_g_and_b}. Recall the sets $\mathcal{S}$ and $\tilde{\mathcal{S}}_{\epsilon}$ from (\ref{definition_of_S}) and (\ref{definition_of_tilde_S}) respectively, and $R_p$ from \eqref{eq:r_p} . Since $ \{\sqrt{p} g^i / \lVert g^i \rVert: 1\leq i \leq m_p \}$ and $\{\lVert g^i \rVert : 1\leq i \leq m_p\}$ are independent, we have
\begin{equation}
\label{G_g}
    \begin{aligned}
        R_p
        =& \frac{1}{p m_p} \ln \left ( \mathbb{E}_{g^1,g^2,\dots,g^{m_p}} \left [ e^{- G(g)} \mathbf{1}_{\{ (g^1 / \lVert g^1 \rVert ,g^2 / \lVert g^2 \rVert, \dots, g^{m_p} / \lVert g^{m_p} \rVert ) \in \mathcal{S}\}} \right ] \right ) \\
        =& \frac{1}{p m_p} \ln \bigg \{ \mathbb{E}_g \left [ e^{- G(g)} \mathbf{1}_{\{ (g^1 / \lVert g^1 \rVert ,g^2 / \lVert g^2 \rVert, \dots, g^{m_p} / \lVert g^{m_p} \rVert ) \in \mathcal{S}\}} \mathbf{1}_{\{ \lVert g^i \rVert \in (\sqrt{p}(1- \epsilon), \sqrt{p} (1 + \epsilon)), \forall i\} }  \right ] \\
        & \quad  \quad  \quad  \quad  \quad  \quad  \quad  \quad  \quad \cdot \mathbb{P}^{-1}(\lVert g^i \rVert \in (\sqrt{p}(1- \epsilon), \sqrt{p} (1 + \epsilon)), \forall i) \bigg \} \\
        =& \frac{1}{p m_p} \ln \left \{ \mathbb{E}_g \left [ e^{- G(g)} \mathbf{1}_{\{ g \in \tilde{\mathcal{S}}_{\epsilon}\}} \right ]     \right \} + O(\epsilon^2)
        \text{,}
    \end{aligned}
\end{equation}
where
\begin{align*}
    G(g) &= G(g^1, g^2,\dots, g^{m_p})  \nonumber \\
    &\coloneqq \frac{m_p}{2\Delta} \left \lVert y - X \left ( \frac{1}{m_p} \sum_{i = 1}^{m_p} \frac{g^i \sqrt{p}}{ \lVert g^i  \rVert}\right ) \right \rVert^2 - \frac{1}{2 \Delta} \sum_{i = 2}^{m_p} \left \lVert X \left ( \frac{g^i \sqrt{p}}{\lVert g^i \rVert} - \frac{1}{m_p} \sum_{i = 1}^{m_p} \frac{g^i \sqrt{p}}{ \lVert g^i  \rVert} \right ) \right \rVert^2 \text{.} \nonumber 
\end{align*}
Define
\begin{equation}
    \tilde{G}(g) = \tilde{G}(g^1, g^2,\dots, g^{m_p})\coloneqq  \frac{m_p}{2\Delta}  \lVert y - X b  \rVert^2 - \frac{1}{2 \Delta} \sum_{i = 2}^{m_p} \left \lVert X(g^i - b)\right \rVert^2 \text{.} \nonumber 
\end{equation}
Note that, for any $ g \in \mathcal{\tilde{S}}_{\epsilon}$, by (\ref{norm_of_b_minus_a}), we have $\lVert b - a \rVert = O(\epsilon \sqrt{p})$ and $\lVert g^i - \beta^i \rVert = O(\epsilon \sqrt{p})$ for any $i$.
Since $\mathbb{P}(\Omega) \to 1$ as $p \to \infty$, with high probability
\begin{equation}
    \begin{aligned}
       & \left | G(g) - \tilde{G}(g) \right |  \\
        \le & \frac{m_p}{2 \Delta} \left | \lVert y - X b  \rVert^2 - \lVert y - X a \rVert^2 \right | + \frac{1}{2 \Delta} \sum_{i = 2}^{m_p} \left |\left \lVert X(g^i - b)\right \rVert^2 - \left \lVert X( \beta^i - a)\right \rVert^2 \right |\\
        \le & \frac{m_p}{2 \Delta} \left | y^{\top} X (b-a)  \right | + \frac{m_p}{2 \Delta} \left | \lVert X b\rVert^2 - \lVert X a \rVert^2 \right | + \frac{1}{2 \Delta} \sum_{i = 2}^{m_p} \bigg \{   \left | (g^i - b)^{\top} X^{\top} X (g^i - b - \beta^i + a)  \right | \\
        & +  \left | (\beta^i - a)^{\top} X^{\top} X (g^i - b - \beta^i + a) \right | \bigg \} \\
        \stackrel{(\mathrm{i})}{\le} & O(\sqrt{p} m_p) \lVert b - a \rVert + \frac{1}{2 \Delta} \sum_{i = 2}^{m_p} \big \{ \lVert X^{\top} X\rVert_2 \big [  \lVert g^i - b \rVert  (\lVert g^i - \beta^i \rVert + \lVert b - a \rVert ) \\
        & + \lVert \beta^i - a \rVert (\lVert g^i - \beta^i \rVert  + \lVert b - a \rVert ) \big ] \big \} \\
        =  & O(\sqrt{p} m_p) \lVert b - a \rVert + O(\sqrt{p}m_p) \lVert g^i - \beta^i \rVert \\
        =  & O(m_p p \epsilon) \text{,} 
    \end{aligned} \nonumber 
\end{equation}
where (i) follows from the fact that $\lVert y \rVert = O(1)$ and $\lVert X \rVert_2 = O(1)$ on $\Omega$. This allows us to replace $G(g)$ in \eqref{G_g} by $\tilde{G}(g)$ with an additional cost $O(\epsilon)$.  Thus with high probability, 
\begin{equation}
\label{R_p_simplified}
    R_p = \frac{1}{p m_p} \ln  \left (   \mathbb{E}_g \left [ e^{- \tilde{G}(g)} \mathbf{1}_{\{ g \in \tilde{\mathcal{S}}_{\epsilon}\}} \right ]  \right ) + O(\epsilon) \text{.}
\end{equation}

Let $\delta_p \coloneqq c_p / C_2 \to 0$, where $C_2$ is a positive constant defined in the statement of Lemma \ref{inner_product_between_g_and_b}. Recalling the RHS of the last equation, by Lemma \ref{inner_product_between_g_and_b}, the definition of $\tilde{\mathcal{S}}_{\epsilon}$ and \eqref{norm_of_b_minus_a}, we know that there exist $c',c>0$ such that  $\{ g = (g^1, g^2, \dots, g^{m_p}): \left | \langle g^i - b, \tilde{b} \rangle \right | \le ( C_2 c_p + C_3 \epsilon)  p, \forall i \} \cap \{ \lVert g^i \rVert \in ((1-\epsilon)\sqrt{p}, (1+ \epsilon)\sqrt{p}), \forall i \} \cap \{ \|b\|/\sqrt{p} : (c', 1-c)\}  $ contains $\tilde{\mathcal{S}}_{\epsilon}$, thus
\begin{align}
 &\mathbb{E}_g \left [ e^{- \tilde{G}(g)} \mathbf{1}_{\{ g \in \tilde{\mathcal{S}}_{\epsilon}\}} \right ] \nonumber \\
        \le & \mathbb{E}_{b} \left [ \mathbf{1}_{\{ \frac{\|b\|}{\sqrt{p}} \in (c', 1-c) \}} \mathbb{E}_{g^2,g^3,\dots,g^{m_p} | b} \left (  e^{- \tilde{G}(g)} \prod_{i = 2}^{m_p}\mathbf{1}_{ \left \{\left | \langle g^i - b, \tilde{b} \rangle \right | \le (C \delta_p + O(\epsilon)) p,  \lVert g^i \rVert \in ((1-\epsilon)\sqrt{p}, (1+ \epsilon)\sqrt{p}) \right \} } \right ) \right ]  \nonumber \\
        = & \mathbb{E}_{b} \left [ 
        \begin{matrix} 
        \mathbf{1}_{\{\frac{\lVert b \rVert}{  \sqrt{p}} \in (c', 1-c) \}}  \times \\ 
         \mathbb{E}_{g^2,g^3,\dots,g^{m_p},w | b} \left (  e^{- \tilde{G}(g)} \prod_{i = 2}^{m_p}\mathbf{1}_{ \left \{\left | \langle g^i - b, \tilde{b} \rangle \right | \le (C \delta_p + O(\epsilon)) p , \lVert g^i \rVert \in ((1-\epsilon)\sqrt{p}, (1+ \epsilon)\sqrt{p}) \right\}} \mathbf{1}_{ \left\{  \lVert w \rVert \le 2\sqrt{ \frac{p}{ m_p}}\right \}  }\right )
         \end{matrix} 
         \right ] \nonumber \\
        & \cdot \mathbb{P}\left (\lVert w \rVert \le 2 \sqrt{\frac{p}{ m_p}}\right )^{-1} \text{,}\label{proof_of_upperbound_intermediate_lemma_2_intermediate_eq_1}
\end{align}
where the last equation follows from the independence of $w$ and $g$. Recall that $h^i = g^i + w$ and $q^i = h^i - b$. We thus have
\begin{equation}
\label{proof_of_upperbound_intermediate_lemma_2_intermediate_eq_2}
    \begin{aligned}
        & \mathbb{E}_{b} \left [ 
         \begin{matrix}
        \mathbf{1}_{\{\frac{\lVert b \rVert}{\sqrt{p}} \in (c', 1-c) \}}  \times \\
        \mathbb{E}_{g^2,g^3,\dots,g^{m_p},w |b } \left (  e^{- \tilde{G}(g)} \prod_{i = 2}^{m_p}\mathbf{1}_{ \left \{\left | \langle g^i - b, \tilde{b} \rangle \right | \le (C \delta_p + O(\epsilon)) p , \lVert g^i \rVert \in ((1-\epsilon)\sqrt{p}, (1+ \epsilon)\sqrt{p}) \right \} } \mathbf{1}_{ \{ \lVert w \rVert \le 2 \sqrt{\frac{p}{m_p}}\}  } \right ) 
        \end{matrix}
         \right ] \\
        \le & \mathbb{E}_{b} \left [ 
        \begin{matrix} 
        \mathbf{1}_{\{\frac{\lVert b \rVert }{\sqrt{p}} \in (c', 1-c) \}} \times \\
         \mathbb{E}_{h^2,h^3,\dots,h^{m_p} | b} \left (  e^{- \tilde{G}(g)} \prod_{i = 2}^{m_p}\mathbf{1}_{ \left \{\left | \langle h^i - b, \tilde{b} \rangle \right | \le (C \delta_p + O(\epsilon) + \frac{2}{\sqrt{m_p}} ) p , \lVert h^i \rVert \in ((1-\epsilon - \frac{2}{\sqrt{m_p}} )\sqrt{p}, (1+ \epsilon + \frac{ 2}{\sqrt{ m_p}})\sqrt{p}) \right \} } \right ) 
         \end{matrix} 
         \right ]\\
        = & \mathbb{E}_{b} \left [ 
        \begin{matrix}
        \mathbf{1}_{\{ \frac{\lVert b \rVert}{ \sqrt{p}} \in (c', 1-c) \}} e^{-\frac{m_p}{2\Delta} \lVert y - X b\rVert^2} \times  \\ 
        \mathbb{E}_{h^2,h^3,\dots,h^{m_p} | b} \left ( e^{  -\frac{1}{2\Delta} \sum_{i=2}^{m_p} \lVert X (h^i - b)\rVert^2 } \prod_{i = 2}^{m_p}\mathbf{1}_{ \left \{\left | \langle h^i - b, \tilde{b} \rangle \right | \le (C \delta_p + O(\epsilon) + \frac{2}{ \sqrt{m_p}} ) p ,  | \lVert h^i \rVert - \sqrt{p}| \le (\epsilon + \frac{2}{\sqrt{ m_p}})\sqrt{p} \right \} } \right ) 
        \end{matrix}
        \right ] \\
        \stackrel{(\mathrm{i})}{=} & \mathbb{E}_{b} \left \{
        \begin{matrix}
        \mathbf{1}_{\{\frac{\lVert b \rVert}{\sqrt{p}} \in (c', 1-c) \}} e^{-\frac{m_p}{2\Delta} \lVert y - X b\rVert^2}\times\\
        \left [ \mathbb{E}_{h^2 | b} \left ( e^{  -\frac{1}{2\Delta}  \lVert X (h^2 - b)\rVert^2 } \mathbf{1}_{ \left \{\left | \langle h^2 - b, \tilde{b} \rangle \right | \le (C \delta_p + O(\epsilon) + \frac{2}{\sqrt{m_p}} ) p , | \lVert h^2 \rVert - \sqrt{p}| \le (\epsilon + \frac{2}{ \sqrt{m_p}})\sqrt{p} \right \} } \right ) \right ]^{m_p -1} 
        \end{matrix}
        \right \} \\
        \le & \mathbb{E}_{b} \left \{
        \begin{matrix} 
        \mathbf{1}_{\{\frac{\lVert b \rVert}{\sqrt{p}} \in (c', 1-c) \}} e^{-\frac{m_p}{2\Delta} \lVert y - X b\rVert^2}\times \\ \left [ \mathbb{E}_{q^2} \left ( e^{  -\frac{1}{2\Delta}  \lVert X q^2 \rVert^2 } \mathbf{1}_{ \left \{\left | \langle q^2, \tilde{b} \rangle \right | \le (C \delta_p + O(\epsilon) + \frac{2}{\sqrt{m_p}} ) p , |\lVert q^2 \rVert^2 - (p - \lVert b \rVert^2) | \le (C \delta_p + O(\epsilon) + \frac{2}{\sqrt{m_p}}) p \right \} } \right ) \right ]^{m_p -1} 
        \end{matrix} 
        \right \}
        \text{.}
    \end{aligned}
\end{equation}
where (i) follows from the fact that $h^i$'s are identically distributed and independent conditioned on $b$. Note that $\epsilon >0 $ is a fixed positive constant. Since $\delta_p \to 0$ and $m_p \to \infty $ as $p \to \infty$, for $p$ large enough we can rewrite the RHS above as
\begin{equation}
\label{proof_of_upperbound_intermediate_lemma_2_intermediate_eq_3}
    \begin{aligned}
         &\mathbb{E}_{b} \left \{ \mathbf{1}_{\{\lVert b \rVert / \sqrt{p} \in (c', 1-c) \}} e^{-\frac{m_p}{2\Delta} \lVert y - X b\rVert^2} \left [ \mathbb{E}_{q^2} \left ( e^{  -\frac{1}{2\Delta} \lVert X q^2 \rVert^2 } \mathbf{1}_{ \left \{\left | \langle q^2, \tilde{b} \rangle \right | \le  O(\epsilon)) p  ,  \lVert q^2 \rVert^2 \in (p - \lVert b \rVert^2  \pm  O(\epsilon)p  ) \right \} }  \right ) \right ]^{m_p -1} \right \} \\
        \le & \sup_{b:\lVert b \rVert / \sqrt{p} \in (c', 1-c)} \left \{ e^{-\frac{m_p}{2\Delta} \lVert y - X b\rVert^2} \left [ \mathbb{E}_{q^2} \left ( e^{  -\frac{1}{2\Delta} \lVert X q^2 \rVert^2 } \mathbf{1}_{ \left \{\left | \langle q^2, \tilde{b} \rangle \right | \le  O(\epsilon)) p  ,  \lVert q^2 \rVert^2 \in (p - \lVert b \rVert^2  \pm  O(\epsilon)p  ) \right \} }  \right ) \right ]^{m_p -1} \right \}
        \text{,}
    \end{aligned}
\end{equation}
where the last line is obtained by Laplace Method. Note that $ \lVert q^2 \rVert$ and $ q^2 / \lVert q^2 \rVert$ are independent. Defining $\chi = \lVert q^2 \rVert$, we further simplify the quantity inside the square bracket above
\begin{equation}
\label{proof_of_upperbound_intermediate_lemma_2_intermediate_eq_4}
\begin{aligned}
        & \mathbb{E}_{q^2} \left ( e^{  -\frac{1}{2\Delta} \lVert X q^2 \rVert^2 } \mathbf{1}_{ \left \{\left | \langle q^2, \tilde{b} \rangle \right | \le  O(\epsilon)) p  , \lVert q^2 \rVert^2 \in (p - \lVert b \rVert^2  \pm  O(\epsilon)p  ) \right \} }  \right) \\
        =& \mathbb{E}_{\beta \sim \operatorname{Unif}(S^{p-1}(\sqrt{p})), \chi \sim \chi(p)} \left (     e^{  -\frac{\chi^2 / p}{2\Delta} \lVert X \beta \rVert^2 } \mathbf{1}_{ \left \{\left | \frac{\chi}{\sqrt{p}} \langle \beta, \tilde{b} \rangle \right | \le  O(\epsilon) p \right \} }  \mathbf{1}_{ \left \{ \chi^2 \in (p - \lVert b \rVert^2  \pm  O(\epsilon)p  ) \right \} }    \right )\\
        \le &  \mathbb{E}_{\beta \sim \operatorname{Unif}(S^{p-1}(\sqrt{p})), \chi \sim \chi(p)} \left (     e^{  -\frac{\chi^2 / p}{2\Delta}  \lVert X \beta \rVert^2 } \mathbf{1}_{ \left \{\left |  \langle \beta, \tilde{b} \rangle \right | \le  O(\epsilon) p \right \} }  \mathbf{1}_{ \left \{ \chi^2 \in (p - \lVert b \rVert^2  \pm  O(\epsilon)p  ) \right \} }    \right ) \\
        \le & \mathbb{E}_{\chi \sim \chi(p)} \left [ \mathbb{E}_{\beta \sim \operatorname{Unif}(S^{p-1}(\sqrt{p}))} \left ( e^{  -\frac{\chi^2 / p}{2\Delta}  \lVert X \beta \rVert^2 } \mathbf{1}_{ \left \{\left |  \langle \beta, \tilde{b} \rangle \right | \le  O(\epsilon) p \right \} }  \right )  \mathbf{1}_{ \left \{ \chi^2 \in (p - \lVert b \rVert^2  \pm  O(\epsilon)p  ) \right \} } \right ]:= \kappa_p(b) \text{.}
    \end{aligned} 
\end{equation}
Finally, combining (\ref{proof_of_upperbound_intermediate_lemma_2_intermediate_eq_1},  \ref{proof_of_upperbound_intermediate_lemma_2_intermediate_eq_2}, \ref{proof_of_upperbound_intermediate_lemma_2_intermediate_eq_3}, \ref{proof_of_upperbound_intermediate_lemma_2_intermediate_eq_4}), we have
\begin{equation}
    \begin{aligned}
   & \frac{1}{p m_p} \ln  \left (   \mathbb{E}_g \left [ e^{- \tilde{G}(g)} \mathbf{1}_{\{ g \in \tilde{\mathcal{S}}_{\epsilon}\}} \right ]  \right ) \nonumber \\ 
    \le &    \frac{1}{p m_p} \ln \left ( \sup_{b \in \mathbb{R}^p, c' \sqrt{p} \le \lVert b \rVert \le (1-c)\sqrt{p}} \left \{ e^{-\frac{m_p}{2\Delta} \lVert y - Xb \rVert^2} \kappa^{m_p - 1}_p(b) \right \} \right )
    -  \frac{1}{p m_p} \ln  \mathbb{P}\Big(\lVert w \rVert \le 2 \sqrt{\frac{p}{m_p}} \Big)^{-1} \\
    = & \sup_{b \in \mathbb{R}^p, c' \sqrt{p} \le \lVert b \rVert \le (1-c)\sqrt{p}} \left \{ \frac{-1}{ 2 p \Delta} \lVert y - Xb \rVert^2 + \frac{m_p-1}{p m_p} \ln \kappa_p(b) \right \}  + o(1) \\
    \le & \sup_{b \in \mathbb{R}^p, c' \sqrt{p} \le \lVert b \rVert \le (1-c)\sqrt{p}} \bigg \{ \frac{-1}{2 p \Delta} \lVert y - Xb \rVert^2 + \frac{1}{p } \ln  \bigg [ \mathbb{E}_{\beta } \bigg ( e^{  -\frac{ (p - \lVert b \rVert^2 - C\epsilon p )/ p}{2\Delta}  \lVert X \beta \rVert^2 } \mathbf{1}_{ \left \{\left |  \langle \beta, \tilde{b} \rangle \right | \le C \epsilon p \right \} }   \bigg )  \bigg ] \\
    & + \frac{1}{p} \ln \mathbb{P}(\chi^2 \in (p - \lVert b \rVert^2  \pm  C \epsilon p  )) \bigg \}  + o(1) \\
    \end{aligned}
\end{equation}
By Lemma \ref{limit_of_free_energy_without_external_field_on_a_linear_subspace} and noting that as $p \to \infty$,
\begin{equation}
     \left | \frac{1}{p} \ln \mathbb{P}(\chi^2 \in (p - \lVert b \rVert^2  \pm  C \epsilon p  )) - \frac{1}{2} \ln \left ( 1 - \frac{\lVert b \rVert^2}{p} \right ) \right | \le   O(\epsilon) + o(1) \nonumber 
\end{equation}
uniformly in $b \in \{b \in \mathbb{R}^p:  c' \sqrt{p} \le \lVert b \rVert \le (1-c)\sqrt{p} \}$, 
we get with high probability
\begin{equation}
    \frac{1}{p m_p} \ln  \left (   \mathbb{E}_g \left [ e^{- \tilde{G}(g)} \mathbf{1}_{\{ g \in \tilde{\mathcal{S}}_{\epsilon}\}} \right ]  \right ) \le \sup_{b \in \mathbb{R}^p , c' \sqrt{p} \le \lVert b \rVert \leq  (1 - c)\sqrt{p}} \{ f_{\text{TAP}}(b) + O(\epsilon) \} \text{.} \nonumber 
\end{equation}
Lemma \ref{upperbound_intermediate_lemma_2} is then established by (\ref{R_p_simplified}).
\end{proof}

\subsection{Restricted free energy and proof of Lemma \ref{lower_and_upper_bounds_on_overlap}}

We turn to the proof of Lemma \ref{lower_and_upper_bounds_on_overlap}.

\begin{proof}[Proof of Lemma \ref{lower_and_upper_bounds_on_overlap}]
First, note that using Lemma \ref{upperbound_intermediate_lemma_1}, there exists $c_p \to 0$ and $m_p \to \infty$ such that for any constant $C>0$, with high probability, for large enough $p$, 
\begin{align}
    \langle \mathbf{1}_{(\beta^1, \cdots, \beta^{m_p}) \in \mathcal{S}} \rangle \geq  \exp(-Cp m_p) >0. \nonumber 
\end{align}
In particular, this implies that $\mathcal{S}$ is nonempty with high-probability. For any $(\beta^1, \cdots, \beta^{m_p}) \in \mathcal{S}$, 
\begin{align}
    0\leq \| \frac{1}{m_p} \sum_i \beta^{i} \|^2 = \frac{p}{m_p^2} + \frac{1}{m_p^2} \sum_{i \neq j} (\beta^i)^{\top} \beta^j \leq \frac{p}{m_p^2} + p(\mathbb{E}[\langle R_{1,2}\rangle] + c_p) \nonumber 
\end{align}
which implies $\mathbb{E}[\langle R_{1,2} \rangle]> - c_p - 1 / (m_p^2)$. This establishes the lower bound. 

Next, we turn to the upper bound. Without loss of generality, we assume
\begin{equation}
    \beta_0 = \sqrt{p} e_1 \text{.} \nonumber 
\end{equation}
Define the \emph{restricted free energy} 
\begin{equation}
\label{eq:def_restricted_free_energy}
    f_p(\delta) = \lim_{\epsilon \to 0^+} \frac{1}{p}  \ln \left \{   \frac{1}{2\epsilon} \int_{ |\beta_1 - \sqrt{p}\delta| < \epsilon}e^{-\frac{1}{2\Delta} \lVert y - X\beta \rVert^2+\mathcal{H}_p^{\text{Gauss}}(\beta)} \mathrm{d} \pi_0(\beta) \right \}.  \nonumber 
\end{equation}
By the Nishimori identity  (see e.g. \cite[(2.25)]{barbier2020strong})
\begin{equation}
\label{eq:R_12_Nishimori}
    \mathbb{E}\langle R_{1,2}\rangle = \frac{1}{p} \mathbb{E} \langle \beta^{\top} \beta_0 \rangle \text{,}
\end{equation}
so it suffices to show there exists $C>0$ such that 
\begin{equation}
\label{eq:R_12_inner_product_wlog}
    (1 - C) > \frac{1}{p} \mathbb{E} \langle \beta^{\top} \beta_0 \rangle = \mathbb{E} \langle \frac{\beta_1}{\sqrt{p}} \rangle, \quad  \text{.}
\end{equation}
\noindent 
The restricted free energy, as defined in \eqref{eq:def_restricted_free_energy}, can be simplified to
\begin{align}
    f_p(\delta) 
    = & \frac{1}{p} \ln \left \{         \frac{1}{\sqrt{\pi}} \frac{\Gamma\left(\frac{p}{2}\right)}{\Gamma\left(\frac{p-1}{2}\right)}\left(1- \delta^{2}\right)^{\frac{p-3}{2}}  \int_{ \beta_1 =\sqrt{p}\delta}e^{-\frac{1}{2\Delta} \lVert y - X\beta \rVert^2+\mathcal{H}_p^{\text{Gauss}}(\beta)} \mathrm{d} \pi_{\beta_1 = \sqrt{p}\delta}(\beta) \right \} \text{,} \nonumber 
\end{align}
where $\pi_{\beta_1 = \sqrt{p}\delta}$ refers to the uniform distribution on the $p-2$ dimensional sphere $\{\beta \in S^{p-1}(\sqrt{p}): \beta_1 = \sqrt{p}\delta \}$. 
Denote $X = [X_1, X_2, \dots, X_p]$. Define $X_{(-1)} \coloneqq [X_2, \dots, X_p]$, $h = \sqrt{p} X_1 + z - \beta_1 X_1$, and $\beta_{(-1)} \coloneqq (\beta_2, \beta_3, \dots, \beta_p)$. Then
\begin{equation}
    y - X \beta = \sqrt{p} X_1 + z - \beta_1 X_1 - X_{(-1)}\beta_{(-1)} = h - X_{(-1)}\beta_{(-1)} \text{.} \nonumber 
\end{equation}
Setting $u \coloneqq \beta_{(-1)} / {\sqrt{p(1-\delta^2)}} \in S^{p-2}(1)$, we have
\begin{equation}
    \begin{aligned}
        \lVert y - X \beta \rVert^2 
        &= \lVert h \rVert^2 + \beta^{\top}_{(-1)} X^{\top}_{(-1)} X_{(-1)} \beta_{(-1)} - 2 h^{\top} X_{(-1)} \beta_{(-1)} \\
        &= \lVert h \rVert^2 + p (1 - \delta^2) u^{\top} X^{\top}_{(-1)} X_{(-1)} u - 2 \sqrt{p (1 - \delta^2)} h^{\top} X_{(-1)} u \text{.}
    \end{aligned} \nonumber 
\end{equation}
With high probability with respect to $Z$, $\sup_{\beta} |\mathcal{H}_p^{\text{Gauss}}(\beta)| =   O(\epsilon_p) \lambda_0 p = o(p)$.
Denoting the uniform measure on $S^{p-2}(1)$ as $\mu_{p-2}$, we have, 
\begin{equation}
    \begin{aligned}
        &f_p(\delta)  - \frac{1}{2} \ln (1 - \delta^2)  \\
      &   =  \frac{1}{p}  \ln \left (  \int_{S^{p-2}(1)} e^{-\frac{1}{2\Delta}\left ( \lVert h \rVert^2 + p (1 - \delta^2) u^{\top} X^{\top}_{(-1)} X_{(-1)} u - 2 \sqrt{p (1 - \delta^2)} h^{\top} X_{(-1)} u  \right )} \mathrm{d} \mu_{p-2}(u)\right ) + o(1) \\
        & = - \frac{ \lVert h \rVert^2}{2p\Delta} + \frac{1}{p} \ln \int_{S^{p-2}(1)} e^{-\frac{1}{2\Delta}\left (  p (1 - \delta^2) u^{\top} X^{\top}_{(-1)} X_{(-1)} u - 2 \sqrt{p (1 - \delta^2)} h^{\top} X_{(-1)} u  \right )} \mathrm{d} \mu_{p-2}(u) + o(1). \label{eq:f_p_exp}  
    \end{aligned}
\end{equation}
By Law of Large Numbers, almost surely
\begin{equation}
    \frac{1}{p} \lVert h \rVert^2 \to \alpha \Delta + (1 - \delta )^2\text{.} \label{eq:h_lim} 
\end{equation}
As $\log (1-\delta^2) \to -\infty$ as $\delta \to 1$, for any $C_1>0$, there exists $\delta <1$ such that with high probability, $\sup_{\delta'\in[\delta,1]} f_p(\delta') < -C_1$. On the other hand, for any fixed $\delta<1$, there exists $C'>0$ (depending on $\alpha$, $\Delta$ and $\delta$) such that with high probability, $\sup_{\delta' \in [-1,\delta]} f_p(\delta') > -C'$. We now proceed as follows: fix any $\delta_0<1$, and assume $\sup_{\delta' \in [-1,\delta_0]} f_p(\delta') > -C_{\delta_0}$. On the other hand, let $\delta_1<1$ be such that with high probability, $\sup_{\delta' \in [\delta_1,1]} f_p(\delta')<-2C_{\delta_0}$. Thus with high-probability, 
\begin{align}
    &\langle \mathbf{1}(\beta_1>\delta_1 \sqrt{p}) \rangle  = \frac{\int_{\delta_1}^1 \exp(pf_p(\delta')) d\delta'}{\int_{-1}^{1} \exp(pf_p(\delta')) d\delta'}
    \leq \frac{\int_{\delta_1}^1 \exp(pf_p(\delta')) d\delta'}{\int_{-1}^{\delta_0} \exp(pf_p(\delta')) d\delta'}  \nonumber \\
    &\leq \frac{(1-\delta_1)\exp(-2C_{\delta_0}p)}{(\delta_0+1) \exp(-p C_{\delta_0}) } \leq \frac{1-\delta_1}{1+\delta_0} \exp(-p C_{\delta_0}). \nonumber 
\end{align}

\noindent 
Thus with probability $1-o(1)$, 
\begin{align}
    \langle \mathbf{1}(\beta_1 <  \delta_1 \sqrt{p} )\rangle > 1- \frac{1-\delta_1}{1+\delta_0} \exp(-p C_{\delta_0}) . \label{eq:req_2} 
\end{align}
This implies, with probability at least $1-o(1)$,  
\begin{align}
    \langle \beta_1 \rangle \leq \delta_1 \sqrt{p} + \sqrt{p} \cdot \frac{1-\delta_1}{1+\delta_0} \exp(-p C_{\delta_0}). \nonumber 
\end{align}
Finally, this implies 
\begin{align}
    \mathbb{E}[\langle \beta_1 \rangle] \leq \delta_1 \sqrt{p} + \sqrt{p} \cdot \frac{1-\delta_1}{1+\delta_0} \exp(-p C_{\delta_0}) + \sqrt{p} o(1) = (\delta_1 +o(1)) \sqrt{p} .  \nonumber 
\end{align}
The required upper bound follows upon setting $C:= (1- \delta_1)/2$. 
\end{proof}

\section{Proof of Theorem \ref{main_theorem}}
\label{sec:combination} 

In this section we prove Theorem \ref{main_theorem} from Theorem \ref{lower_bound} and Theorem \ref{upper_bound} by arguing the small constant $c$  within Theorem \ref{lower_bound} has negligible impact.

\begin{proof}[Proof of Theorem \ref{main_theorem}]
First, the upper bound, namely 
\begin{equation}
    \mathbb{P} \left (  \frac{1}{p} \ln \mathcal{Z}_p \ge \sup_{a \in \mathbb{R}^p, \lVert a \rVert \le \sqrt{p}} f_{\text{TAP}}(a) + \eta  \right ) \to 0 \text{,} \nonumber 
\end{equation}
follows immediately from Theorem \ref{upper_bound}. 
For the matching lower bound, for any $a \in \mathbb{R}^p$, $\|a\| \leq \sqrt{p}$, define 
$$a^{\star} \coloneqq \argmin_{x \in \mathbb{R}^p,  \lVert x \rVert \le (1-c)\sqrt{p} }\lVert a - x \rVert,$$ then
\begin{equation}
\label{eq:main_theorem_proof_eq1}
    \begin{aligned}
        &  \sup_{a \in \mathbb{R}^p, \lVert a \rVert \le  \sqrt{p}} f_{\text{TAP}}(a) - \sup_{a \in \mathbb{R}^p,  \lVert a \rVert \le (1-c)\sqrt{p}} f_{\text{TAP}}(a)  \\
        \le &  \max \left \{0,  \sup_{a \in \mathbb{R}^p, (1 - c) \sqrt{p} \le \lVert a \rVert \le   \sqrt{p}} \left ( f_{\text{TAP}}(a) - f_{\text{TAP}}(a^\star)    \right )  \right \}\text{,}
    \end{aligned}
\end{equation}
the second term of which can be further controlled by
    \begin{align*}
        & \sup_{a \in \mathbb{R}^p, (1 - c) \sqrt{p} \le \lVert a \rVert \le   \sqrt{p}} \left ( f_{\text{TAP}}(a) - f_{\text{TAP}}(a^\star)    \right ) \nonumber \\
        \le & \sup_{a \in \mathbb{R}^p, (1 - c) \sqrt{p} \le \lVert a \rVert \le   \sqrt{p}}  \Big \{ \frac{1}{2 \Delta p}  \left | \lVert y - Xa \rVert^2 - \lVert y -  X a^\star \rVert^2  \right |  \nonumber \\
        & - \frac{\alpha}{2} \left [ \ln \left (1 + \frac{1 - \lVert a \rVert^2 / p}{\Delta \alpha} \right ) - \ln \left (1 + \frac{1 - \lVert a^\star \rVert^2 / p}{\Delta \alpha} \right )\right ] + \frac{1}{2} \left [ \ln \left  (1 - \frac{\lVert a \rVert^2}{p} \right) - \ln \left  (1 - \frac{\lVert a^\star \rVert^2}{p} \right) \right ] \Big \} \nonumber \\
        \le &\sup_{a \in \mathbb{R}^p, (1 - c) \sqrt{p} \le \lVert a \rVert \le   \sqrt{p}} \Big \{ \frac{1}{2 \Delta p}  \left [ \left | 2 y^{\top} X (a - a^\star) \right | + \left |a^{\top} X^{\top} X a - (a^{\star})^{\top} X^{\top} X a^{\star} \right | \right ] \Big \} \nonumber \\
        & + \sup_{a \in \mathbb{R}^p, (1 - c) \sqrt{p} \le \lVert a \rVert \le   \sqrt{p}} \Big \{ \frac{\alpha}{2} \left [ \ln \left (1 + \frac{1 - \lVert a^\star \rVert^2 / p}{\Delta \alpha} \right ) - \ln \left (1 + \frac{1 - \lVert a \rVert^2 / p}{\Delta \alpha} \right )\right ]   \Big \} \\
        & + \sup_{a \in \mathbb{R}^p, (1 - c) \sqrt{p} \le \lVert a \rVert \le   \sqrt{p}} \Big \{ \frac{1}{2} \left [ \ln \left  (1 - \frac{\lVert a \rVert^2}{p} \right) - \ln \left  (1 - \frac{\lVert a^\star \rVert^2}{p} \right) \right ] \Big \} \\
        \le & \sup_{a \in \mathbb{R}^p, (1 - c) \sqrt{p} \le \lVert a \rVert \le   \sqrt{p}} \left \{ \frac{1}{2 \Delta p}  \left [  2 c \sqrt{p} \lVert y^{\top} X  \rVert + 2 p  ( 2 c- c^2) \lVert X^{\top} X \rVert_2 \right ] + \tau(c) \right \} \text{,}
    \end{align*}
where 
$$
\tau(c) \coloneqq \frac{\alpha}{2} \ln \left ( 1 + \frac{2c - c^2}{\Delta \alpha}\right )\text{,}
$$
as a deterministic function,  goes to $0$ as $c \to 0$.
Since with high probability, $\lVert X^{\top} X \rVert_2 = O(1)$ and $ \lVert y^{\top} X \rVert = O(\sqrt{p})$, we have
\begin{equation}
    \sup_{a \in \mathbb{R}^p, (1 - c) \sqrt{p} \le \lVert a \rVert \le   \sqrt{p}} \left ( f_{\text{TAP}}(a) - f_{\text{TAP}}(a^\star)    \right )  \to 0 \nonumber 
\end{equation}
as $c \to 0$, with high probability. Thus, there exists positive constant $c^\star$ such that if $c \le c^\star$, then with high probability,
\begin{equation}
\label{eq:main_theorem_proof_eq2}
   \sup_{a \in \mathbb{R}^p, (1 - c) \sqrt{p} \le \lVert a \rVert \le   \sqrt{p}} \left ( f_{\text{TAP}}(a) - f_{\text{TAP}}(a^\star)    \right )  \le \frac{\eta}{3} \text{.}
\end{equation}
Finally, choosing $c\le c^\star $, Theorem \ref{lower_bound} together with (\ref{eq:main_theorem_proof_eq1}) and (\ref{eq:main_theorem_proof_eq2}) give us
\begin{equation}
    \mathbb{P} \left (  \frac{1}{p} \ln \mathcal{Z}_p \le \sup_{a \in \mathbb{R}^p, \lVert a \rVert \le \sqrt{p}} f_{\text{TAP}}(a) - \frac{\eta}{3}  \right ) \to 0 \text{,} \nonumber 
\end{equation}
which is exactly the lower bound part of Theorem \ref{main_theorem}.
\end{proof}


\begin{appendix}


In Appendix \ref{sec:conc} we collect some standard technical results that are used in our proofs. In the same vein, Appendix \ref{sec:covariance} collects some facts about sample covariance matrices with iid gaussian entries. Finally, Appendix \ref{sec:overlap_conc} establishes Theorem \ref{overlap_concentration}. 

\section{Concentration and convexity inequalities}
\label{sec:conc} 
\begin{lemma}[Concentration of almost Lipschitz functions \cite{barbier2020mutual}]
\label{concentration_of_almost_Lipshitz_functions}
Let $U_i \overset{\text{iid}}{\sim} N(0,1)$. Let $f$ be a Lipschitz function over $G \subset \mathbb{R}^{p}$ with Lipschitz constant $K_{p}$. Assume $0 \in G, f(0)^{2} \leq C^{2}, \mathbb{E}\left[f^{2}\right] \leq C^{2}$ for some $C>0 .$ Then for $r \geq 6\left(C+\sqrt{p K_{p}}\right) \sqrt{\mathbb{P}\left(G^{c}\right)}$ we have
$$
\mathbb{P}\left(\left|f\left(U_{1}, \ldots, U_{p}\right)-\mathbb{E}\left[f\left(U_{1}, \ldots, U_{p}\right)\right]\right| \ge r\right) \leq 2 e^{-\frac{r^{2}}{16 K_p^{2}}}+\mathbb{P}\left(G^{c}\right) \text{.}
$$
\end{lemma}
\begin{lemma}[The largest and smallest eigenvalues of Wishart matrices \cite{vershynin2010introduction}]
\label{large_deviation_characterization_for_the_largest_eigenvalue_of_Weshart_matrices}
Let $A$ be an $n \times p$ matrix whose entries are independent standard normal random variables. Then for every $t \ge 0$, with probability at least $1-2 \exp \left(-t^{2} / 2\right)$ one has
$$
\sqrt{n}-\sqrt{p} - t \le \sigma_{\min }(A) \le \sigma_{\max }(A) \le \sqrt{n} + \sqrt{p} + t \text{.}
$$
\end{lemma}
\begin{lemma}[L\'{e}vy's lemma \cite{vershynin2018high}]
\label{levys_lemma}
Let $S^{n-1}(1)$ be the $(n-1)$-sphere of radius 1. Let $f: S^{n-1}(1) \rightarrow \mathbb{R}$ be a Lipschitz function with Lipschitz constant $L$, and let $\sigma$ be a uniform random vector on $S^{n-1}(1)$. Then for any $t > 0$,
$$
\mathbb{P}(|f(\sigma)-\mathbb{E} f(\sigma)| \ge L t) \le \exp \left(\pi-n t^{2} / 4\right).
$$
\end{lemma}
\begin{lemma}[A bound for differences between derivatives of convex functions \cite{barbier2019optimal}]
    \label{lemma_bound_for_convex_functions}
    Let $G(x)$ and $g(x)$ be convex functions. Let $\delta > 0$ and define $C_{\delta}^{-}(x):=g^{\prime}(x)-g^{\prime}(x-\delta) \ge 0$ and $C_{\delta}^{+}(x):=g^{\prime}(x+\delta)-g^{\prime}(x) \ge 0 .$ Then
$$
\left|G^{\prime}(x)-g^{\prime}(x)\right| \le \delta^{-1} \sum_{u \in\{x-\delta, x, x+\delta\}}|G(u)-g(u)|+C_{\delta}^{+}(x)+C_{\delta}^{-}(x) \text{.}
$$
\end{lemma}


\section{Facts about the random matrix \texorpdfstring{$X^{\top} X$}{} }
\label{sec:covariance} 

In this section, we list some well-known facts about the random matrix $X^{\top} X$, which is usually referred to as \textit{sample covariance matrix} in statistics and probability literature. Please refer to \cite{bai2010spectral} for complete statements and proofs. Let $\theta_1 < \theta_2 < \dots < \theta_p$ be the eigenvalues of the matrix $S_p = X^{\top} X \in \mathbb{R}^{p \times p}$. We have that as $p \to \infty$, almost surely,
\begin{equation}
    \frac{1}{p} \sum_{i = 1}^{p} \delta_{\theta_i} \to \mu(x) dx \text{ in distribution,} \nonumber 
\end{equation}
where $\mu(\cdot)$ is the Marchenko-Pastur law. Moreover, if we let $\theta_t$ be the $t$-quantile of $\mu$, then
\begin{equation}
    \theta^i = \theta_{i/p } + o(1) \text{ for } i \in \{1,2,\dots, p\} \text{,}  \nonumber 
\end{equation}
where the $o(1)$ terms above converge to $0$ almost surely, uniformly in $i$.

\section{Proof of overlap concentration results}
\label{sec:overlap_conc} 

\subsection{Proof of Theorem \ref{overlap_concentration}}

First of all, please note that in this subsection we will use $\mathbb{E}$ as a shorthand for $\mathbb{E}_{X,z,Z,\beta_0}$. Recall that $\mathcal{H}_p^{\text{Gauss}}$ is the Hamiltonian of the side information channel, namely
\begin{equation}
    \mathcal{H}_p^{\text{Gauss}} =  - \lambda_0 \epsilon_p \beta_0^{\top} \beta - \sqrt{\lambda_0 \epsilon_p} Z^{\top} \beta + \frac{1}{2} \lambda_0 \epsilon_p \lVert \beta \rVert^2 \text{,} \nonumber 
\end{equation}
which can be seen as only keeping the $\beta$-related terms in $\lVert Y^{\text{Gauss}  }-\sqrt{\lambda_0 \epsilon_p} \beta \rVert^2 / 2$. Let $\lambda_{0,p} = \lambda_0\epsilon_p$. Furthermore, let
\begin{equation}
\label{definition_of_H_prime}
    H' \coloneqq \frac{ \mathrm{d} \mathcal{H}_p^{\text{Gauss}}}{ \mathrm{d} \lambda_{0,p}} =  - \beta_0^{\top} \beta - \frac{Z^{\top} \beta}{2 \sqrt{\lambda_{0,p}}} + \frac{\lVert \beta \rVert^2}{2} \text{.}
\end{equation}
We begin the proof of overlap concentration, i.e. Theorem \ref{overlap_concentration}, by introducing the following lemma, which upper bounds the fluctuations of $R_{1,2}$ by the fluctuations of $\mathcal{L} \coloneqq H' / p$.
\begin{lemma}
\label{control_fluctuations_of_R_12_by_that_of_L}
    Let 
    $$
    \mathcal{L} \coloneqq \frac{H'}{p} = \frac{1}{p} \frac{\mathrm{d} \mathcal{H}_p^{\text{Gauss}}}{ \mathrm{d} \lambda_{0,p}} \text{.}
    $$
    Then
    \begin{equation}
        \mathbb{E} \left \langle (R_{1,2} - \mathbb{E}\langle R_{1,2} \rangle)^2 \right \rangle \le 4 \mathbb{E} \left \langle (\mathcal{L} - \mathbb{E}\langle \mathcal{L} \rangle)^2 \right \rangle \text{.} \nonumber 
    \end{equation}
\end{lemma}
\begin{proof}
    Let $R_{1,*} = \beta_0^{\top} \beta / p$. Using the definition of $H'$ gives 
    \begin{equation}
        \begin{aligned}
            2 \mathbb{E}\left\langle R_{1, *}(\mathcal{L}-\mathbb{E}\langle\mathcal{L}\rangle)\right\rangle = \mathbb{E}\left[\frac{1}{p}\right.&\left.\left\langle R_{1, *}\|\beta\|^{2}\right\rangle-2\left\langle R_{1, *}^{2}\right\rangle - \frac{1}{p \sqrt{\lambda_{0, p}}}\left\langle R_{1, *} Z^{\top} \beta\right\rangle\right] \\
            &-\mathbb{E}\left\langle R_{1, *}\right\rangle \mathbb{E}\left[\frac{1}{p}\left\langle\|\beta \|^{2}\right\rangle-2\left\langle R_{1, *}\right\rangle-\frac{1}{p \sqrt{\lambda_{0, p}}} Z^{\top} \langle \beta \rangle\right] \text{.}
        \end{aligned} \nonumber 
    \end{equation}
    Gaussian integration by parts implies
    \begin{equation}
        \frac{1}{p \sqrt{\lambda_{0, p}}} \mathbb{E}\left\langle R_{1, *} Z^{\top} \beta\right\rangle
        = \frac{1}{p} \mathbb{E}\left\langle R_{1, *}\|\beta\|^{2}\right\rangle-\frac{1}{p} \mathbb{E}\left\langle R_{1, *} \beta^{\top} \langle \beta \rangle \right \rangle
        =\frac{1}{p} \mathbb{E}\left\langle R_{1, *}\|\beta\|^{2}\right\rangle - \mathbb{E}\left [ \left\langle R_{1, *}\right\rangle^{2} \right ] \text{,} \nonumber 
    \end{equation}
    in which $\mathbb{E}\left\langle R_{1, *} \beta^{\top} \langle \beta \rangle \right \rangle = \frac{1}{p} \mathbb{E} \left [ \langle \beta_0^{\top} \beta \beta^{\top} \rangle \langle \beta \rangle \right ] = \frac{1}{p}  \mathbb{E} \left [ \langle \beta_0^{\top} \beta^1 (\beta^1)^{\top} \beta^2 \rangle \right ] = \frac{1}{p}  \mathbb{E} \left [ \langle (\beta^1)^{\top} \beta_0 \beta_0^{\top} \beta^2 \rangle \right ] = p \mathbb{E} \left [ \left\langle R_{1, *}\right\rangle^{2} \right ]$ is due to the Nishimori identity. Again, using Gaussian integration by parts 
    \begin{equation}
        \frac{1}{p\sqrt{\lambda_{0, p}}}\mathbb{E}\langle Z^{\top} \beta\rangle = \frac{1}{p} \mathbb{E}\left\langle\|\beta\|^{2}\right\rangle-\mathbb{E}\left\langle R_{1, *}\right\rangle \text{.} \nonumber 
    \end{equation}
    Putting them together, we have 
    \begin{equation}
        \begin{aligned}
            2 \mathbb{E}\left\langle R_{1, *}(\mathcal{L}-\mathbb{E}\langle\mathcal{L}\rangle)\right\rangle & = \mathbb{E}\langle R_{1,*}\rangle^{2}-2 \mathbb{E}\langle R_{1, *}^{2}\rangle+(\mathbb{E}\langle R_{1, *}\rangle)^{2} \\
            & = -(\mathbb{E}\langle R_{1, *}^{2}\rangle-(\mathbb{E}\langle R_{1, *}\rangle)^{2})-(\mathbb{E}\langle R_{1, *}^{2}\rangle-\mathbb{E}\langle R_{1, *}\rangle^{2}) \\
            & \le - \mathbb{E} \left \langle (R_{1,2} - \mathbb{E}\langle R_{1,2} \rangle)^2 \right \rangle\\
            & \le 0 \text{.}
        \end{aligned} \nonumber 
    \end{equation}
    Finally, using Cauchy-Schwarz inequality, we have
    \begin{equation}
        \begin{aligned}
            \mathbb{E} \left \langle (R_{1,2} - \mathbb{E}\langle R_{1,2} \rangle)^2 \right \rangle &\le 2 | \mathbb{E} \left\langle R_{1, *}(\mathcal{L}-\mathbb{E}\langle\mathcal{L}\rangle)\right\rangle  |\\
            &= 2 | \mathbb{E} \left\langle (R_{1, *} - \mathbb{E} \langle R_{1,*} \rangle ) ( \mathcal{L} - \mathbb{E} \langle \mathcal{L} \rangle) \right \rangle  |\\
            &\le 2 \left \{ \mathbb{E} \left\langle (R_{1, *} - \mathbb{E} \langle R_{1,*} \rangle )^2 \right \rangle \mathbb{E} \left \langle ( \mathcal{L} - \mathbb{E} \langle \mathcal{L} \rangle)^2 \right \rangle \right \}^{1/2} \text{,}
        \end{aligned} \nonumber 
    \end{equation}
    which implies
    \begin{equation}
        \mathbb{E} \left \langle (R_{1,2} - \mathbb{E}\langle R_{1,2} \rangle)^2 \right \rangle = \mathbb{E} \left \langle (R_{1,*} - \mathbb{E}\langle R_{1,*} \rangle)^2 \right \rangle \le 4 \mathbb{E} \left \langle (\mathcal{L} - \mathbb{E}\langle \mathcal{L} \rangle)^2 \right \rangle \text{.}\nonumber 
    \end{equation} 
\end{proof}
Since the above lemma basically says the fluctuations of the overlaps are upper bounded by those of $\mathcal{L}$, it is sufficient to prove that $\mathcal{L}$ will concentrate. Before moving on, let us recall a few notations. As in (\ref{eq:perturbed_posterior}), $\mathcal{H}^{\text{Pert}}$ is the Hamiltonian of the whole perturbed model, i.e. including both the original model and the side information channel, and $F_{p}^{\text{Pert}}$ is the corresponding free energy, as defined in (\ref{eq:perturbed_free_energy}). 
Moreover, let
\begin{equation}
    \mathcal{H}' \coloneqq \frac{ \mathrm{d} \mathcal{H}^{\text{Pert}}}{ \mathrm{d} \lambda_{0,p}} \text{.} \nonumber 
\end{equation}
Please note $\mathcal{H}'$ and $H'$ defined in (\ref{definition_of_H_prime}) are essentially the same since $\mathcal{H} = \mathcal{H}^{\text{Pert}} - \mathcal{H}^{\text{Gauss}}$ does not depend on $\lambda_{0,p}$.
Armed with these notations, we observe, 
\begin{equation}
    \frac{\mathrm{d} F_{p}^{\text{Pert }}}{\mathrm{d} \lambda_{0, p}}= - \frac{1}{p} \left\langle\mathcal{H}^{\prime}\right\rangle, \quad \text { and } \quad \frac{\mathrm{d}^{2} F_{p}^{\text{Pert }}}{\mathrm{d} \lambda_{0, p}^{2}}= \frac{1}{p} \left \{ \left\langle\left(\mathcal{H}^{\prime}-\left\langle\mathcal{H}^{\prime}\right\rangle\right)^{2}\right\rangle-\frac{\langle\beta\rangle^{\top}  Z}{4 \lambda_{0, p}^{3 / 2}} \right \} \text{.} \nonumber 
\end{equation}
Exchanging expectation and derivatives, we obtain
\begin{align*}
   & \frac{\mathrm{d} \mathbb{E} F_{p}^{\text{Pert }}}{\mathrm{d} \lambda_{0, p}}= - \frac{1}{p} \mathbb{E}\left \langle \mathcal{H}^{\prime} \right \rangle =  - \frac{1}{2} \mathbb{E} \left \langle R_{1,2} \right \rangle,  \\
    &\frac{ \mathrm{d}^{2} \mathbb{E} F_{p}^{\text{Pert }}}{\mathrm{d} \lambda_{0, p}^{2}}= \frac{1}{p} \left \{ \mathbb{E}\left\langle\left(\mathcal{H}^{\prime}-\left\langle\mathcal{H}^{\prime}\right\rangle\right)^{2}\right\rangle-\frac{1}{4 \lambda_{0, p}} \mathbb{E}\left\langle\|\beta-\langle\beta\rangle\|^{2}\right\rangle \right \}, \nonumber 
\end{align*}
in which we have utilized both Nishimori identity and the following identity obtained using Gaussian integration by parts
    \begin{equation}
        \frac{1}{p\sqrt{\lambda_{0, p}}}\mathbb{E}\langle Z^{\top} \beta\rangle = \mathbb{E}\left\langle\|\beta\|^{2}\right\rangle-\mathbb{E}\left\langle R_{1, *}\right\rangle \text{.} \nonumber 
    \end{equation}

\begin{lemma}[Fluctuations of $\mathcal{L}$]
\label{fluctuations_of_L}
Let $\lambda_0 \sim \operatorname{Unif}[1/2,1]$. If $\frac{v_p}{p  \epsilon_p} \to 0$ then there exists a constant $C > 0$ such that
\begin{equation}
    \mathbb{E}_{\lambda_0}\mathbb{E} \left \langle (\mathcal{L} - \mathbb{E}\langle \mathcal{L} \rangle)^2 \right \rangle \le \frac{C}{\epsilon_p} (\frac{v_p}{p\epsilon_p} + \frac{1}{p})^{1/3} \text{.} \nonumber 
\end{equation}
\end{lemma}
\begin{proof}
    To prove this lemma, we consider the following decomposition and split our proof into controlling these two terms respectively,
    \begin{equation}
        \mathbb{E}\left\langle(\mathcal{L}-\mathbb{E}\langle\mathcal{L}\rangle)^{2}\right\rangle=\mathbb{E}\left\langle(\mathcal{L}-\langle\mathcal{L}\rangle)^{2}\right\rangle+\mathbb{E}(\langle\mathcal{L}\rangle-\mathbb{E}\langle\mathcal{L}\rangle)^{2}\text{.} \nonumber 
    \end{equation}
    We start with the first term. Observe that 
    \begin{align*}
       & \int_{\epsilon_p /2}^{\epsilon_p}  \mathbb{E}\left \langle ( \mathcal{H}' - \langle \mathcal{H}' \rangle)^2 \right \rangle \mathrm{d}\lambda_{0,p} \\
        &\le \int_{\epsilon_p /2}^{\epsilon_p} \left ( p \frac{\mathrm{\mathrm{d}}^{2} \mathbb{E} F_{p}^{\text{Pert }}}{\mathrm{d} \lambda_{0, p}^{2}} + \frac{p}{\lambda_{0,p}}     \right )\mathrm{d}\lambda_{0,p} = p \left ( \left. \frac{\mathrm{d} \mathbb{E} F_{p}^{\text{Pert }}}{\mathrm{d} \lambda_{0, p}}\right|_{\lambda_{0, p}=\varepsilon_{p} / 2} ^{\lambda_{0, p}=\varepsilon_{p}} \right ) + p\ln 2 \text{.} \nonumber 
    \end{align*}
    Since
    \begin{equation}
        \left | \frac{\mathrm{d} \mathbb{E} F_{p}^{\text{Pert }}}{\mathrm{d}\lambda_{0, p}} \right | = \frac{1}{p} | \mathbb{E}\left\langle\mathcal{H}^{\prime}\right\rangle | = \left | \frac{1}{2} \mathbb{E} \left \langle R_{1,2} \right \rangle \right | \le \frac{1}{2} \text{,} \nonumber 
    \end{equation}
    For any $\lambda_{0,p}$, we naturally have a trivial bound of $\left.\frac{\mathrm{d} \mathbb{E} F_{p}^{\text{Pert }}}{\mathrm{d} \lambda_{0, p}}\right|_{\lambda_{0, p}=\varepsilon_{p} / 2} ^{\lambda_{0, p}=\varepsilon_{p}}$, i.e.,
    \begin{equation}
        \left.\frac{\mathrm{d} \mathbb{E} F_{p}^{\text{Pert }}}{\mathrm{d} \lambda_{0, p}}\right|_{\lambda_{0, p}=\varepsilon_{p} / 2} ^{\lambda_{0, p}=\varepsilon_{p}} \le 1 \text{.} \nonumber 
    \end{equation}
    Recalling that $\mathcal{L} = H' / p =\mathcal{H}' /p$,
    \begin{equation}
        \mathbb{E}_{\lambda_0}\mathbb{E}\left\langle(\mathcal{L}-\mathbb{E}\langle\mathcal{L}\rangle)^{2}\right\rangle = \frac{1}{p^2} \mathbb{E}_{\lambda_0} \mathbb{E}\left\langle(\mathcal{H'}-\mathbb{E}\langle\mathcal{H'}\rangle)^{2}\right\rangle \le 
        \frac{1 + \ln 2}{p }  \text{.} \nonumber 
    \end{equation}
    Now we proceed to control the second term using convexity arguments. 

    We introduce the following two convex functions of $\lambda_{0,p}$, 
    \begin{equation}
        \widetilde{F}\left(\lambda_{0, p}\right):=F_{p}^{\text{Pert }}\left(\lambda_{0, p}\right) -  \sqrt{\frac{\lambda_{0, p}}{p}} \lVert Z \rVert \text{,} \nonumber 
    \end{equation}
    and
    \begin{equation}
        \mathbb{E} \widetilde{F}\left(\lambda_{0, p}\right):=   \mathbb{E} F_{p}^{\text{Pert }}\left(\lambda_{0, p}\right)- \sqrt{\frac{\lambda_{0, p}}{p}} \mathbb{E}\lVert Z \rVert \text{.} \nonumber 
    \end{equation}
    Letting $A_{p}:=p^{-1/2} \left ( \lVert Z \rVert - \mathbb{E} \lVert Z \rVert \right )$, we have 
    $$
    \quad \widetilde{F}\left(\lambda_{0, p}\right)-\mathbb{E} \widetilde{F}\left(\lambda_{0, p}\right)=\left(F_{p}^{\text{Pert }}\left(\lambda_{0, p}\right)-\mathbb{E} F_{p}^{\text{Pert }}\left(\lambda_{0, p}\right)\right)-\sqrt{\lambda_{0, p}} A_{p}
    \text{,}
    $$
    as well as
    $$
    \widetilde{F}^{\prime}\left(\lambda_{0, p}\right)-\mathbb{E} \widetilde{F}^{\prime}\left(\lambda_{0, p}\right)=\langle\mathcal{L}\rangle-\mathbb{E}\langle\mathcal{L}\rangle-\frac{A_{p}}{2 \sqrt{\lambda_{0, p}}}
    \text{.}
    $$
    Thus Lemma \ref{lemma_bound_for_convex_functions} implies
    $$
    |\langle\mathcal{L}\rangle-\mathbb{E}\langle\mathcal{L}\rangle| 
    \le
    \delta^{-1} \sum_{u \in \mathcal{U}}\left(\left|F_{p}^{\text{Pert }}(u)-\mathbb{E} F_{p}^{\text{Pert }}(u)\right|+\left|A_{p}\right| \sqrt{u}\right)+C_{\delta}^{+}\left(\lambda_{0, p}\right)+C_{\delta}^{-}\left(\lambda_{0, p}\right)+\frac{\left|A_{p}\right|}{2 \sqrt{\lambda_{0, p}}} \text{,}
    $$
    where $\mathcal{U}:=\left\{\lambda_{0, p}-\delta, \lambda_{0, p}, \lambda_{0, p}+\delta\right\}$ and
    $$
    C_{\delta}^{-}\left(\lambda_{0, p}\right):=\mathbb{E} \widetilde{F}^{\prime}\left(\lambda_{0, p}\right)-\mathbb{E} \widetilde{F}^{\prime}\left(\lambda_{0, p}-\delta\right) \ge 0, \quad C_{\delta}^{+}\left(\lambda_{0, p}\right):=\mathbb{E} \widetilde{F}^{\prime}\left(\lambda_{0, p}+\delta\right)-\mathbb{E} \widetilde{F}^{\prime}\left(\lambda_{0, p}\right) \ge 0 \text{.}
    $$
    Note that by using \cite[Theorem 3.1.1]{vershynin2018high}, there exists a positive constant $C_Z$ such that $\mathbb{E}\left [ \left ( \lVert Z \rVert - \mathbb{E} \lVert Z \rVert \right )^2 \right ] \le C_Z$ independently of $p$, which ensures $\mathbb{E} \left ( A_p^2 \right ) \le C_Z/ p = O(1/p)$. Thus, by the definition of $v_p$ in the statement of Theorem \ref{overlap_concentration},
    \begin{equation}
        \begin{aligned}
            \mathbb{E}(\langle\mathcal{L}\rangle-\mathbb{E}\langle\mathcal{L}\rangle)^2 
            \le 9 & \mathbb{E} \Bigg \{   \delta^{-2} \sum_{u \in \mathcal{U}}\left [ \left|F_{p}^{\text{Pert}}(u)-\mathbb{E} F_{p}^{\text {pert }}(u)\right|^2+(\left|A_{p}\right| \sqrt{u})^2 \right ] \\
            & + C_{\delta}^{+}\left(\lambda_{0, p}\right)^2 + C_{\delta}^{-}\left(\lambda_{0, p}\right)^2 +  \left ( \frac{\left|A_{p}\right|}{2 \sqrt{\lambda_{0, p}}} \right )^2 \Bigg \} \\
            \le & 9  \left \{ \frac{3}{p \delta^2}[v_p + (\lambda_{0,p} + \delta) C_Z ] +  C_{\delta}^{+}\left(\lambda_{0, p}\right)^2 + C_{\delta}^{-}\left(\lambda_{0, p}\right)^2 + \frac{C_Z}{4 p \lambda_{0,p}}\right \} \text{.}
        \end{aligned} \nonumber 
    \end{equation}
    Note that 
    \begin{equation}
        | \mathbb{E} \widetilde{F}^{\prime}\left(\lambda_{0, p}\right) | = \left |  \frac{1}{2} \mathbb{E} R_{1,2}    + \frac{\mathbb{E} \lVert Z \rVert }{2\sqrt{\lambda_{0,p} p}} \right | 
        \le \frac{1}{2} \left (1 + \frac{1}{\sqrt{\lambda_{0,p}}} \right ) \text{,} \nonumber 
    \end{equation}
    and thus
    \begin{equation}
        \left|C_{\delta}^{\pm}\left(\lambda_{0, p}\right)\right| \le 1+\frac{1}{\sqrt{\varepsilon_{p} / 2-\delta}} \text{,} \nonumber 
    \end{equation}
    for $\delta < \epsilon_p / 2$, which further gives us
    \begin{equation}
        \begin{array}{r}
        \int_{\epsilon_{p} / 2}^{\epsilon_{p}} d \lambda_{0, p}\left\{C_{\delta}^{+}\left(\lambda_{0, p}\right)^{2}+C_{\delta}^{-}\left(\lambda_{0, p}\right)^{2}\right\} 
        \le
        \left(1+\frac{1}{\sqrt{\epsilon_{p} / 2-\delta}}\right) \int_{\epsilon_{p} / 2}^{\epsilon_{p}} d \lambda_{0, p}\left\{C_{\delta}^{+}\left(\lambda_{0, p}\right)+C_{\delta}^{-}\left(\lambda_{0, p}\right)\right\} \\
        =\left(1+\frac{1}{\sqrt{\epsilon_{p} / 2-\delta}}\right)\left[\left(\mathbb{E} \widetilde{F}\left(\epsilon_{p} / 2+\delta\right)-\mathbb{E} \widetilde{F}\left(\epsilon_{p} / 2-\delta\right)\right)+\left(\mathbb{E} \widetilde{F}\left(\varepsilon_{p}-\delta\right)-\mathbb{E} \widetilde{F}\left(\varepsilon_{p}+\delta\right)\right)\right] \text{.}
        \end{array} \nonumber 
    \end{equation}
    By the mean value theorem,
    \begin{equation}
        |\mathbb{E} \widetilde{F}\left(\epsilon_{p} / 2+\delta\right)-\mathbb{E} \widetilde{F}\left(\epsilon_{p} / 2-\delta\right)| \le 2 \delta \max_{x \in [\epsilon_{p} / 2 - \delta, \epsilon_{p} / 2+\delta]}{\mathbb{E} \widetilde{F}'(x)} \le \delta \left (1 + \frac{1}{\sqrt{\epsilon_p / 2 - \delta}}  \right ). \nonumber 
    \end{equation}
    Similarly,
    \begin{equation}
        |\mathbb{E} \widetilde{F}\left(\epsilon_{p}+\delta\right)-\mathbb{E} \widetilde{F}\left(\epsilon_{p} -\delta\right)| \le 2 \delta \max_{x \in [\epsilon_{p}  - \delta, \epsilon_{p}+\delta]}{\mathbb{E} \widetilde{F}'(x)} \le \delta \left (1 + \frac{1}{\sqrt{\epsilon_p  - \delta}}  \right ) \le \delta \left (1 + \frac{1}{\sqrt{\epsilon_p / 2 - \delta}}  \right ). \nonumber 
    \end{equation}
    Therefore
    \begin{equation}
        \int_{\epsilon_{p} / 2}^{\epsilon_{p}} \mathrm{d} \lambda_{0, p}\left\{C_{\delta}^{+}\left(\lambda_{0, p}\right)^{2}+C_{\delta}^{-}\left(\lambda_{0, p}\right)^{2}\right\} \le 2 \delta \left (1 + \frac{1}{\sqrt{\epsilon_p / 2 - \delta}}  \right )^2 \le 4\delta \frac{\epsilon_p /2 - \delta + 1 }{\epsilon_p / 2 - \delta} \le \frac{8 \delta}{\epsilon_p /2 - \delta} \nonumber 
    \end{equation}
    where we set $\delta = \delta_p $ so that ${\delta_p} / {\epsilon_p} \to 0$ and ${\delta_p} / {\epsilon_p} < 1/4 $ for any $p$.
    Putting them all together, we have
    \begin{equation}
        \begin{aligned}
        \int_{\epsilon_{p} / 2}^{\epsilon_{p}} \mathrm{d} \lambda_{0, p} \mathbb{E}(\langle\mathcal{L}\rangle-\mathbb{E}\langle\mathcal{L}\rangle)^2 
        & \le \frac{27 \epsilon_p}{2p\delta_p^2 }(v_p + 2 \epsilon_p C_Z)  + 9 \int_{\epsilon_p/2}^{\epsilon_p} \frac{C_Z}{4p\lambda_{0,p}} \mathrm{d} \lambda_{0,p} + \frac{288 \delta_p}{\epsilon_p} \\
        & = \frac{27 \epsilon_p}{2p\delta_p^2 }(v_p + 2 \epsilon_p C_Z) + \frac{288 \delta_p}{\epsilon_p} + \frac{9 C_Z \ln 2}{4p}  \text{.}
        \end{aligned} \nonumber 
    \end{equation}

    Finally, we choose $\delta_p = \left \{ \left [ 3\epsilon_p^2 (v_p + 2 \epsilon_p C_Z) \right  ] / (64 p) \right \}^{1/3}$ to optimize the bound, which together with the previous bound on $\mathbb{E}_{\lambda_0}\mathbb{E} \left\langle(\mathcal{L}-\mathbb{E}\langle\mathcal{L}\rangle)^{2}\right\rangle$ would result in
    \begin{equation}
        \begin{aligned}
            \mathbb{E}_{\lambda_0}\mathbb{E} \left \langle (\mathcal{L} - \mathbb{E}\langle \mathcal{L} \rangle)^2 \right \rangle  
            & \le \frac{576}{\epsilon_p} \left ( \frac{3 \epsilon_p^2 ( v_p + 2 \epsilon_p C_Z)}{64p}\right )^{1/3} + \frac{9 C_Z \ln 2}{4p} +  \frac{1 + \ln 2}{p } \\
            & \le C \left (\frac{v_p}{p\epsilon_p} + \frac{1}{p} \right )^{1/3} 
            \text{,}
        \end{aligned} \nonumber 
    \end{equation}
    for some positive constant $C$.
\end{proof}
Combining Lemma \ref{control_fluctuations_of_R_12_by_that_of_L} and Lemma \ref{fluctuations_of_L} gives us the following lemma, which is essentially the first half of Theorem \ref{overlap_concentration}.
\begin{lemma}[Overlap concentration: Fluctuations of $R_{1,2}$]
\label{overlap_concentration_R_12_fluctuations}
    Assuming $\lambda_0 \sim \operatorname{Unif}(\frac{1}{2},1)$, there exist a positive constant $C$ such that
    \begin{equation}
        \mathbb{E}_{\lambda_0} \mathbb{E} \langle (R_{1,2} - \mathbb{E}\langle R_{1,2} \rangle )^2 \rangle \le C \Big(\frac{v_p}{p\epsilon_p} + \frac{1}{p} \Big)^{1/3} \text{,} \nonumber 
    \end{equation}
    where $v_p \coloneqq p \sup_{\lambda_0 \in [\frac{1}{2},1]} \{ \mathbb{E} (F_p^{\text{Pert}}(\lambda_0) -  \mathbb{E}F_p^{\text{Pert}}(\lambda_0))^2  \}$.
\end{lemma}




\subsection{Concentration of the free energy with respect to the Gaussian quenched variables}

\begin{lemma}
    \label{Concentration of the free energy with respect to the Gaussian quenched variables}
    There exist two positive constants $C_1$ and $C_2$ (depending only on $\Delta, \alpha$) such that for any $r \ge O(e^{- C_1 p^{1/2}})$, we have for any fixed $\beta_0$ and $\lambda_0 \in [1/2,1]$,
    \begin{equation}
        \mathbb{P}(\frac{1}{p}|F^{\text{Pert}}_p - \mathbb{E}[F^{\text{Pert}}_p|\beta_0]  | \ge r | \beta_0)  \le e^{-C_2 r^2 p^{1/2}} \text{.} \nonumber 
    \end{equation}
\end{lemma}
\begin{proof}
    Consider
    \begin{equation}
        \mathcal{G} = \{X,z,Z: \max _{\mu}\left|z_{i}\right| \leq \sqrt{D_{1}}\ \cap \ \max _{i}\left|Z_{i}\right| \leq \sqrt{D_{1}}\ \cap \  \text { for all } \beta, \beta_0:\|X(\beta-\beta_0)\|^{2} \leq D_{2} p \} 
        \text{.} \nonumber 
    \end{equation}
    In order to use the known results regarding concentration of almost Lipschitz functions, we first try to control $\mathbb{P}(\mathcal{G}^c)$. Since if $U \sim N(0,1)$ then $\mathbb{P}(|U| \ge \sqrt{A}) \le 4 e^{-A/4}$, we have by union bound
    \begin{equation}
        \mathbb{P}(\max _{\mu}\left|z_{i}\right| \leq \sqrt{D_{1}}\ \cap \ \max _{i}\left|\widetilde{z}_{i}\right| \leq \sqrt{D_{1}}) \le 4 p (1+ \alpha ) e^{-D_1 /4} 
        \text{,} \nonumber 
    \end{equation}
    and
    \begin{equation}
        \mathbb{P}(\text{for all } \beta, \beta_0:\|X(\beta-\beta_0)\|^{2} \le D_{2} p ) = \mathbb{P} (\Lambda_1 (X^{\top} X) \le D_2 / 4)
        \text{.}  \nonumber 
    \end{equation}
    Thus
    \begin{equation}
        \mathbb{P}(\mathcal{G}^c) \le  4 p (1+ \alpha ) e^{-D_1 /4} + \left (  1-  \mathbb{P} \left (\sigma_{\text{max}}(X)^2 \le \frac{D_2}{4} \right ) \right )
        \text{,} \nonumber 
    \end{equation}
    which will be made small with a suitable choices of $D_1$ and $D_2$.
    Now we proceed to show, for any fixed $\beta_0$, free energy of the perturbed model $F^{\text{Pert}}_p(X,z,Z,\beta_0)$ is Lipschitz on $\mathcal{G}$ with respect to $(X, z, Z)$, as a vector in $\mathbb{R}^{np+n+p}$. From now on, we treat $\beta_0$ as fixed and omit it in the arguments of $\mathcal{H}^{\text{Pert}}$ and $F^{\text{Pert}}$. Recall that, as defined in (\ref{eq:perturbed_free_energy}),
    \begin{equation}
        \mathcal{H}^{\text{Pert}}(X, z, Z) = \frac{1}{2\Delta} \lVert y - X\beta \rVert^2 - \lambda_0 \epsilon_p \beta_0^{\top} \beta - \sqrt{\lambda_0 \epsilon_p} Z^{\top} \beta + \frac{1}{2} \lambda \epsilon_p \lVert \beta \rVert^2 \text{.} \nonumber 
    \end{equation}
    Then
    \begin{equation}
        \begin{aligned}
            | \mathcal{H}^{\text{Pert}}(X, z, Z) - \mathcal{H}^{\text{Pert}}(X', z', Z') | \le \frac{1}{2\Delta} \left | \lVert X(\beta_0 - \beta) + z \rVert^2 - \lVert X'(\beta_0 - \beta) + z' \rVert^2 \right | \\
             + \sqrt{\lambda \epsilon_p} \left | Z^{\top} \beta - (Z')^{\top} \beta \right | \text{,}
        \end{aligned} \nonumber 
    \end{equation}
    in which 
    \begin{equation}
        \begin{aligned}
            &\left | \lVert X(\beta_0 - \beta) + z \rVert^2 - \lVert X'(\beta_0 - \beta) + z' \rVert^2 \right | \\
            =& \left | \left (\lVert X(\beta_0 - \beta)\rVert^2 + 2 z^{\top} X(\beta_0 - \beta) + \lVert z \rVert^2 \right )  - \left(\lVert X'(\beta_0 - \beta)\rVert^2 + 2 z'^{\top} X'(\beta_0 - \beta) + \lVert z' \rVert^2 \right )\right | \\
            \le &  | (\beta - \beta_0)^{\top} (X^{\top} X - X'^{\top} X') (\beta - \beta_0)  | + 2 |(z^{\top} X - z'^{\top} X') (\beta - \beta_0)| + |\lVert z \rVert^2 - \lVert z' \rVert^2| \\
            \le & \lVert \beta - \beta_0 \rVert_2^2 \left ( \lVert X^{\top}(X - X') \rVert + \lVert (X - X')^{\top} X' \rVert \right ) + (O(\sqrt{pD_1} ) + O(\sqrt{p D_2})) (\sqrt{p} \lVert X - X' \rVert_F + \lVert z - z' \rVert)\\
            \le &(O(\sqrt{pD_1} ) + O(\sqrt{p D_2})) (\sqrt{p} \lVert X - X' \rVert_F + \lVert z - z' \rVert)
            \text{,}
        \end{aligned} \nonumber
    \end{equation}
    and
    \begin{equation}
        \begin{aligned}
        \left | Z^{\top} \beta - (Z')^{\top} \beta \right | \le \sqrt{\lambda_0 \epsilon_p p}  \lVert Z - Z' \rVert \text{.}
        \end{aligned} \nonumber 
    \end{equation}
    Since $\lambda_0 \in [1/2,1]$ and $\lim_{p \to \infty} \epsilon_p = 0 $, we get, as long as both $D_1$ and $D_2$ are always lower bounded by a positive constant,  
    \begin{equation}
        | \mathcal{H}^{\text{Pert}}(X,z,Z) - \mathcal{H}^{\text{Pert}}(X',z',Z') | \le (O(\sqrt{pD_1} ) + O(\sqrt{p D_2})) (\sqrt{p} \lVert X - X' \rVert_F + \lVert z - z' \rVert + \lVert Z - Z' \rVert) \text{.} \nonumber 
    \end{equation}
    For the difference of free energies, we thus have
    \begin{equation}
        \begin{aligned}
        & F^{\text{Pert}}_p(X,z,Z) - F^{\text{Pert}}_p(X',z',Z')  \\
        &= \frac{1}{p} \ln \left ( \frac{\int \mathrm{d} \beta \pi_{0}(\beta) e^{ - \mathcal{H}^{\text{Pert}}_\lambda(y)}}{\int \mathrm{d} \beta \pi_{0}(\beta) e^{ - \mathcal{H}^{\text{Pert}}_\lambda(y')}} \right )\\
        & \le \frac{1}{p} \ln \left ( \frac{\int \mathrm{d} \beta \pi_{0}(\beta) e^{ - \mathcal{H}^{\text{Pert}}_\lambda(y') + |\mathcal{H}^{\text{Pert}}_\lambda(y) - \mathcal{H}^{\text{Pert}}_\lambda(y')|}}{\int \mathrm{d} \beta \pi_{0}(\beta) e^{ - \mathcal{H}^{\text{Pert}}_\lambda(y')}} \right ) \\
        & \le \frac{1}{p} (O(\sqrt{pD_1} ) + O(\sqrt{p D_2})) (\sqrt{p} \lVert X - X' \rVert_F + \lVert z - z' \rVert + \lVert Z - Z' \rVert) \text{.}
        \end{aligned} \nonumber 
    \end{equation}
    By symmetry, we can further conclude that 
    \begin{equation}
        \begin{aligned}
       & \left | F^{\text{Pert}}_p(X,z,Z) - F^{\text{Pert}}_p(X',z',Z') \right | \\
        & \le  (O(p^{-1/2}\sqrt{D_1} ) + O(p^{-1/2}\sqrt{D_2})) (\sqrt{p} \lVert X - X' \rVert_F + \lVert z - z' \rVert + \lVert Z - Z' \rVert)\\
        & \le (O(p^{-1/2} \sqrt{D_1} ) + O(p^{-1/2} \sqrt{D_2})) \lVert (\sqrt{p}X,z,Z) - (\sqrt{p}X',z',Z') \rVert \text{,}
        \end{aligned} \nonumber 
    \end{equation}
    where with a slight abuse of notation both $(\sqrt{p}X,z,Z)$ and $(\sqrt{p}X',z',Z')$ are treated as vectors in $\mathbb{R}^{np+n+p}$. In order to apply Lemma \ref{concentration_of_almost_Lipshitz_functions}, the last thing we need to show is that $\mathbb{E}_{X,z,Z} \left [ \left (  F_p^{\text{Pert}}\right )^2 \right ] = \mathbb{E} \left [ \left (  F_p^{\text{Pert}}\right )^2 \big | \beta_0 \right ]$ is bounded. In fact, the upper bound constant does not depend on $\beta_0$.
    \begin{equation}
        \begin{aligned}
            \left | F_p^{\text{Pert}} \right | 
            &= \left | \frac{1}{p} \ln \int_{S^{p-1}(\sqrt{p})} e^{-\frac{1}{2\Delta} \lVert y - X\beta \rVert^2 - \lambda \epsilon_p \beta_0^{\top} \beta - \sqrt{\lambda \epsilon_p} Z^{\top} \beta + \frac{1}{2} \lambda \epsilon_p \lVert \beta \rVert^2} \mathrm{d} \pi_0(\beta) \right  | \\
            &\le \left | \frac{1}{p} \ln \int_{S^{p-1}(\sqrt{p})} e^{- \frac{1}{2\Delta} \lVert X(\beta_0 - \beta) + z \rVert^2 - \lambda \epsilon_p \left | \beta_0^{\top} \beta \right | - \sqrt{\lambda \epsilon_p} \left | Z^{\top} \beta \right | - \frac{1}{2} \lambda \epsilon_p \lVert \beta \rVert^2} \mathrm{d} \pi_0(\beta) \right  | \\
            &\le \left | \frac{1}{p} \ln \int_{S^{p-1}(\sqrt{p})} e^{- \frac{1}{\Delta} ( \lVert X(\beta_0 - \beta) \rVert^2  + \lVert z \rVert^2) - p \lambda \epsilon_p - \sqrt{p \lambda \epsilon_p} \lVert Z \rVert  - \frac{1}{2} p\lambda \epsilon_p } \mathrm{d} \pi_0(\beta) \right  | \\
            &\le \frac{1}{p} \left [    \frac{4p}{\Delta} \lVert X^{\top} X\rVert_2 + \frac{1}{\Delta} \lVert z \rVert^2 + \frac{3}{2}p \lambda \epsilon_p + \sqrt{p \lambda \epsilon_p} \lVert Z \rVert   \right ] \text{,}
        \end{aligned} \nonumber 
    \end{equation}
    which leads to
    \begin{equation}
        \begin{aligned}
            \mathbb{E} \left [ \left (F_p^{\text{Pert}}(\lambda)\right )^2 \right ]
            &\le \mathbb{E} \left  \{ \frac{4}{p^2} \left [ \frac{16p^2}{\Delta^2} \sigma_{\text{max}}(X)^4 + \frac{1}{\Delta^2} \lVert z \rVert^4 + \frac{9 \lambda \epsilon_p}{4} p^2 + p \lambda \epsilon_p \lVert Z \rVert^2 \right ]\right \} = O(1) \text{.}
        \end{aligned} \nonumber 
    \end{equation}
    By applying lemma \ref{concentration_of_almost_Lipshitz_functions}  
    and choosing $D_1 = p^{\gamma}, D_2 = O(1), K_p = O(p^{-(1-\gamma)/2})$ (accordingly $\mathbb{P}(\mathcal{G}^c) = O(p e^{- p^{\gamma} / 4})$), for any $r \ge O(e^{- C_1 p^{\gamma}}) \ge O(p^{(5+\gamma)/4}e^{-p^{\gamma}/8})$,
    \begin{equation}
        \mathbb{P}(|F^{\text{Pert}}_p(\lambda) - \mathbb{E}[F^{\text{Pert}}_p(\lambda)|\beta_0]  | \ge r | \beta_0)  \le O(e^{- C r^2 p^{1-\gamma}}) + O(p e^{-p^{\gamma}/4})
        \text{.} \nonumber 
    \end{equation}
Finally, the choice $\gamma = 1/2$ renders the optimal bound
\begin{equation}
    \mathbb{P}(|F^{\text{Pert}}_p(\lambda) - \mathbb{E}[F^{\text{Pert}}_p(\lambda)|\beta_0]  | \ge r | \beta_0)  \le e^{-C_2 r^2 p^{1/2}}
    \text{.} \nonumber 
\end{equation}
\end{proof}
We will see that, with an integral argument, lemma \ref{Concentration of the free energy with respect to the Gaussian quenched variables} further implies the following lemma, which gives a upper bound on the conditional variance of $F^{\text{Pert}}_p(\lambda)$ giving any fixed $\beta_0$.
\begin{lemma}
For any fixed $\beta_0 \in S^{p-1}(\sqrt{p})$,
    \begin{equation}
        \mathbb{E} \left [ \left ( F_p^{\text{Pert}}(\lambda) - \mathbb{E} \left [ F_p^{\text{Pert}}(\lambda) \big | \beta_0 \right ]\right )^2 \bigg | \beta_0 \right ] = O \left ( \frac{1}{p} \right )
        \text{.}  \nonumber 
    \end{equation}
\end{lemma}

\begin{proof}
    By Lemma \ref{Concentration of the free energy with respect to the Gaussian quenched variables}, there exist positive constants $C_1, C_2,$ and $C_3$, depending only on $\Delta$ and $\alpha$, such that for any $r \ge t_0 \coloneqq C_1 e^{- C_2 p^{1/2}}$, we have for any $\beta_0 \in S^{p-1}(\sqrt{p})$,
    \begin{equation}
        \mathbb{P} \left (  \left | F_p^{\text{Pert}}(\lambda) - \mathbb{E}F_p^{\text{Pert}}(\lambda) \right |  > r  \bigg | \beta_0  \right ) \le e^{-C_3 r^2 p^{1/2}} \text{.} \nonumber 
    \end{equation}
    Let $V = \left | F_p^{\text{Pert}}(\lambda) - \mathbb{E}F_p^{\text{Pert}}(\lambda) \right | / p $. 
    \begin{equation}
        \begin{aligned}
            \mathbb{E}\left [ V^2  \big | \beta_0 \right ]
            &= \mathbb{E}\left [ V^2  \mathbb{1}_{V \le t_0}  \big | \beta_0  \right ] + \mathbb{E}\left [ V^2  \mathbb{1}_{V > t_0}  \big | \beta_0  \right ]\\
            &\le t_0^2 \mathbb{P}(V \le t_0  \big | \beta_0 ) + \int_{t_0^2}^{\infty} s \mathbb{P}(V^2 > s  \big | \beta_0  ) \mathrm{d} s \\
            &\le t_0^2 + \int_{t_0^2}^{\infty} s e^{-C_3 s p^{1/2}} \mathrm{d} s\\
            &= t_0^2 + \frac{1}{C_3 p^{1/2}} \left ( t_0^2 + \frac{1}{C_3 p^{1/2}} \right ) e^{-C_3 p^{1/2} t_0^2} \\
            &= O \left ( \frac{1}{p} \right )
            \text{,}
        \end{aligned} \nonumber 
    \end{equation}
    where the last equality is because of $t_0^2 = C_1^2 e^{- 2 C_2 p^{1/2}} = o(p^{-1/2})$.
\end{proof}


\subsection{Concentration of the free energy with respect to the signal}

In this subsection, we state and prove Lemma \ref{Concentration of the free energy with respect to the signal}, which is an upper bound on $\mathbb{E} \left [ F_p^{\text{Pert}}(\lambda) \big | \beta_0 \right]$, or equivalently, we prove that free energy of the perturbed model concentrates with respect to the randomness of $\beta_0$.
\begin{lemma}
\label{Concentration of the free energy with respect to the signal}
    \begin{equation}
        \mathbb{E} \left [ \left ( \mathbb{E} \left [ F_p^{\text{Pert}}(\lambda) \big | \beta_0 \right] - \mathbb{E} F_p^{\text{Pert}}(\lambda) \right )^2 \right ] = O \left ( \frac{1}{p} \right )
        \text{.} \nonumber 
    \end{equation}
\end{lemma}
\begin{proof}
  By Jensen's inequality, we have
\begin{equation}
    \begin{aligned}
        \frac{1}{p} \langle \mathcal{H}_{\lambda}^{\text{Pert}}(\tilde{\beta}_0) - \mathcal{H}_{\lambda}^{\text{Pert}}(\beta_0)  \rangle_{\tilde{{\beta}}_0} 
        \le \frac{1}{p} (\ln \mathcal{Z}^{\text{Pert}}(\beta_0) - \ln \mathcal{Z}^{\text{Pert}}(\tilde{\beta}_0)) 
        \le \frac{1}{p} \langle \mathcal{H}_{\lambda}^{\text{Pert}}(\tilde{\beta}_0) - \mathcal{H}_{\lambda}^{\text{Pert}}(\beta_0)  \rangle_{{\beta}_0}
        \text{,}
    \end{aligned} \nonumber 
\end{equation}
where $\langle f(\beta) \rangle_{\beta_0}$ refers to the expectation of $f(\beta)$ under the distribution proportional to $ \mathrm{d} \pi(\beta) e^{-\mathcal{H}^{\text{Pert}}(\beta,\beta_0)}$ and similarly for $\langle f(\beta) \rangle_{\tilde{\beta}_0}$. Note that $\mathcal{H}^{\text{Pert}}(\beta,\beta_0)$ is the same as $\mathcal{H}^{\text{Pert}}(\beta)$, only to emphasize the different values of $\beta_0$. Moreover,
\begin{equation}
    \begin{aligned}
       & \left | \mathbb{E}_{X,z,Z} ( \mathcal{H}_{\lambda}^{\text{Pert}}(\beta, \tilde{\beta}_0) - \mathcal{H}_{\lambda}^{\text{Pert}}(\beta, \beta_0) ) \right |  \nonumber \\
        & \le \mathbb{E}_{X,z,Z} \bigg \{  \frac{1}{2\Delta} \left | \lVert X(\beta_0 - \beta) + z \rVert^2 - \lVert X(\tilde{\beta}_0 - \beta) + z \rVert^2 \right |  + \frac{1}{2} \left | \lambda_0 \epsilon_p (\beta_0 - \tilde{\beta}_0)^{\top} \beta   \right | \bigg \} \\
        & \le \mathbb{E}_{X,z,Z} \bigg \{ \frac{1}{2 \Delta} \left (\lVert X(\beta_0 - \tilde{\beta}_0)\rVert^2 + 2 \left | \langle X(\beta_0 - \tilde{\beta}_0), X(\tilde{\beta}_0 - \beta) + z \rangle \right | \right )  + O(\sqrt{p}) \lVert \beta_0 - \tilde{\beta}_0 \rVert \bigg \}\\
        &= O(\sqrt{p}) \lVert \beta_0 - \tilde{\beta}_0 \rVert
        \text{.}
    \end{aligned}
\end{equation}
Upon defining
\begin{equation}
    f(\frac{1}{\sqrt{p}} \beta_0 ) \coloneqq \frac{1}{p}\mathbb{E}_{X,z,Z} \ln \mathcal{Z}^{\text{Pert}}(\beta_0) = \mathbb{E}_{X,z,Z} F_p^{\text{Pert}} \text{,} \nonumber 
\end{equation}
we see that there exists a positive constant $C$ such that
\begin{equation}
    \left | f(\frac{1}{\sqrt{p}} \beta_0 ) - f(\frac{1}{\sqrt{p}} \tilde{\beta}_0 ) \right | \le C \lVert \frac{1}{\sqrt{p}}\beta_0 - \frac{1}{\sqrt{p}}\tilde{\beta}_0 \rVert
    \text{.} \nonumber 
\end{equation}
Finally, L\'{e}vy's lemma, i.e. Lemma \ref{levys_lemma}, implies
\begin{equation}
    \mathbb{P}\left (  \left | \mathbb{E}_{X,z,Z} F_p^{\text{Pert}} - \mathbb{E}_{X,z,Z,\beta_0} F_p^{\text{Pert}} \right | > C t \right ) \le e^{\pi - p t^2 / 4} \text{,} \nonumber 
\end{equation}
for any $t > 0$, which further leads to
\begin{equation*}
    \mathbb{E} \left [ \left ( \mathbb{E} \left [F_p^{\text{Pert}} \big | \beta_0 \right ] - \mathbb{E} F_p^{\text{Pert}} \right )^2 \right ] \le \frac{4 C^2 e^{\pi}}{p} = O \left ( \frac{1}{p} \right )
    \text{.}
\end{equation*}
\end{proof}



\subsection{Completing the proof of Theorem \ref{overlap_concentration}}

\begin{proof}[Proof of Theorem \ref{overlap_concentration}]

By Lemma \ref{Concentration of the free energy with respect to the Gaussian quenched variables} and \ref{Concentration of the free energy with respect to the signal},
\begin{equation*}
    \begin{aligned}
        \operatorname{Var}(F_{p}^{\text {Pert }}(\lambda)) 
        = & \mathbb{E} \left [ \operatorname{Var}( F_{p}^{\text {Pert }}(\lambda) \big | \beta_0 )\right ] + \operatorname{Var} \left ( \mathbb{E}[F_{p}^{\text {Pert }}(\lambda) \big | \beta_0]    \right )\\
        \le & O \left ( \frac{1}{p} \right )+ O \left ( \frac{1}{p} \right ) = O \left ( \frac{1}{p} \right ) \text{.}
    \end{aligned}
\end{equation*}
Recall that
\begin{equation*}
    \frac{\mathrm{d} \mathbb{E} F_{p}^{\text {Pert }}}{\mathrm{d} \lambda_{0, p}}= \frac{1}{p} \mathbb{E}\left\langle\mathcal{H}^{\prime}\right\rangle = \frac{1}{2} \mathbb{E}\left\langle R_{1,2}\right\rangle \text{,}
\end{equation*}
which gives
\begin{equation*}
    \left | \frac{\mathrm{d} \mathbb{E} F_{p}^{\text {Pert }}}{\mathrm{d} \lambda_{0}} \right | = \left | \epsilon_p \frac{\mathrm{d} \mathbb{E} F_{p}^{\text {Pert }}}{\mathrm{d} \lambda_{0, p}} \right | = \left | \frac{\epsilon_p}{p} \mathbb{E}\left\langle\mathcal{H}^{\prime} \right \rangle \right | =  \left | \frac{\epsilon_p}{2} \mathbb{E}\left\langle R_{1,2}\right\rangle \right | \le \frac{ \epsilon_p}{2} \text{.}
\end{equation*}
Finally, noting that when $\lambda_0 = 0$, the perturbed model would be the same as our original linear regression model. Thus,
\begin{equation*}
    \left | \mathbb{E} F_{p}^{\text {Pert }}(\lambda_0)  -  \mathbb{E} F_{p}  \right | =  \left | \mathbb{E} F_{p}^{\text {Pert }}(\lambda_0)  -  \mathbb{E} F_{p}^{\text {Pert }}(0)  \right | \le  \frac{\lambda_0 \epsilon_p}{2}  = O(\epsilon_p),
\end{equation*}
since $\lambda_0 \in [1/2,1 ]$. Together with Lemma \ref{overlap_concentration_R_12_fluctuations}, the proof is therefore complete.
\end{proof}

\end{appendix}
%
%

\begin{acks}[Acknowledgments]
SS thanks Sumit Mukherjee for his encouragement during the completion of this manuscript. 
\end{acks}
\begin{funding}
%
SS was partially supported by a Harvard Dean's Competitive Fund Fellowship. 
\end{funding}



\bibliographystyle{imsart-number} 
\bibliography{references}       


\end{document}